\documentclass[10pt]{amsart}
      \usepackage[mathscr]{eucal}
      \usepackage{amsmath,amsfonts}
      \parskip\smallskipamount
          \newtheorem{theorem}{Theorem}[section]
      \newtheorem{definition}[theorem]{Definition}
      \newtheorem{proposition}[theorem]{Proposition}
      \newtheorem{corollary}[theorem]{Corollary}

      \makeatletter
      \@addtoreset{equation}{section}
      \makeatother

      \newcommand{\BB}{{\mathbb B}}
      \newcommand{\CC}{{\mathbb C}}
      \newcommand{\NN}{{\mathbb N}}
      
      \newcommand{\ZZ}{{\mathbb Z}}
      \newcommand{\DD}{{\mathbb D}}
      \newcommand{\RR}{{\mathbb R}}
      \newcommand{\FF}{{\mathbb F}}
      \newcommand{\TT}{{\mathbb T}}

      \newcommand{\cA}{{\mathcal A}}
      
      \newcommand{\cC}{{\mathcal C}}
      \newcommand{\cD}{{\mathcal D}}
      \newcommand{\cE}{{\mathcal E}}

      \newcommand{\cH}{{\mathcal H}}
      
      \newcommand{\cK}{{\mathcal K}}

      \newcommand{\cM}{{\mathcal M}}

      \newcommand{\cP}{{\mathcal P}}
      \newcommand{\cR}{{\mathcal R}}
      \newcommand{\cS}{{\mathcal S}}

      \newcommand{\cY}{{\mathcal Y}}
      \newcommand{\cX}{{\mathcal X}}

      \newdimen\expt
      \def\boxit#1{\setbox0\hbox{$\displaystyle{#1}$}
            \hbox{\lower.4\expt
       \hbox{\lower3\expt\hbox{\lower\dp0
            \hbox{\vbox{\hrule height.4\expt
       \hbox{\vrule width.4\expt\hskip3\expt
            \vbox{\vskip3\expt\box0\vskip2\expt}%
       \hskip3\expt\vrule width.4\expt}\hrule height.4\expt}}}}}}
      
\begin{document}
       \pagestyle{myheadings}
      \markboth{ Gelu Popescu}{Multi-Toeplitz operators and free pluriharmonic functions}

      \title [Multi-Toeplitz operators and free pluriharmonic functions]
      { Multi-Toeplitz operators and free pluriharmonic functions
      }
        \author{Gelu Popescu}
\date{November 15, 2018}
     \thanks{Research supported in part by  NSF grant DMS 1500922}
      \address{Department of Mathematics, The University of Texas
      at San Antonio \\ San Antonio, TX 78249, USA}
      \email{\tt gelu.popescu@utsa.edu}

\begin{abstract}
We initiate the study of weighted multi-Toeplitz operators associated with noncommutative regular domains ${\bf D}_q^m(\cH)\subset  B(\cH)^n$, $m, n\geq 1$, where $B(\cH)$ is the algebra of all bounded linear operators on a Hilbert space $\cH$. These operators are acting on the full Fock space with $n$ generators and have as symbols free pluriharmonic functions on the interior of the domain  ${\bf D}_q^m(\cH)$.
We prove that the set of all weighted multi-Toeplitz operators coincides with
 $$
\overline{\cA({\bf D}_q^m)^*+\cA({\bf D}_q^m)}^{\text\rm WOT},
$$
where the domain algebra $\cA({\bf D}_q^m)$ is the norm-closed unital non-selfadjoint algebra generated by the universal model $(W_1,\ldots, W_n)$  of the noncommutative domain ${\bf D}_q^m(\cH)$. These results are used to study the class of  free pluriharmonic functions on    ${\bf D}_q^m(\cH)^\circ$. Several classical results from complex analysis  concerning harmonic functions have analogues in our noncommutative setting. In particular, we show that the bounded  free pluriharmonic functions are precisely those  which are noncommutative Berezin transforms of weighted multi-Toeplitz operators, and solve the Dirichlet extension problem in this setting.
Using noncommutative Cauchy transforms, we provide a free analytic functional calculus for $n$-tuples of  operators, which extends to free pluriharmonic functions.
Our study of weighted multi-Toeplitz operators on Fock spaces is a blend of  multi-variable operator theory, noncommutative function theory, operator spaces, and harmonic analysis.
\end{abstract}

      \maketitle

\bigskip

\section*{Introduction}

Let $H^2(\DD)$ be the Hardy space of all analytic functions on the open unit disc
$\DD:=\{z\in \CC: \ |z|<1\}$ with square-sumable coefficients. An operator $T\in B(H^2(\DD))$ is called Toeplitz if
$$Tf=P_+(\varphi f),\qquad f\in H^2(\TT),
$$
 for some $\varphi \in L^\infty(\TT)$, where $P_+$ is the orthogonal projection of the Lebesgue space $L^2(\TT)$ onto the Hardy space $H^2(\TT)$, which is identified with $H^2(\DD)$.
 Brown and Halmos \cite{BH} proved that a necessary and sufficient condition that an operator on the Hardy space $H^2(\DD)$ be a Toeplitz operator is that its matrix
 $[\lambda_{ij}]$ with respect to the standard basis $e_k(z)=z^k$, $k\in \{0,1,\ldots\}$, be a Toeplitz  matrix, i.e
 $$
 \lambda_{i+1,j+1}=\lambda_{ij},\qquad  i,j\in \{0,1,\ldots\},
 $$
 which is equivalent to $S^*TS=T$, where $S$ is the unilateral shift on $H^2(\DD)$.
 In this case, $\lambda_{ij}=a_{i-j}$, where $\varphi =\sum_{k\in \ZZ}a_k \chi_k$ is the Fourier expansion of the symbol $\varphi\in L^\infty(\TT)$. The class of Toeplitz operators originates with O.~Toeplitz \cite{T} and has been studied extensively over the years, starting with Hartman and Wintner \cite{HW} and  the seminal paper of Brown and Halmos \cite{BH}.  The study of Toeplitz operators on the Hardy space $H^2(\DD)$ was extended to Hilbert spaces of holomorphic functions  on the unit disc (see \cite{HKZ}) such as the Bergman space and weighted Bergman space,
 and also to higher dimensional setting  involving holomorphic functions in several complex variables on various classes of domains in $\CC^n$ (see Upmeier's book \cite{U}).

 The class of Toeplitz operators is one of the most important classes of non-selfadjoint operators having    applications in  index theory and noncommutative geometry, prediction theory, boundary values problems for analytic functions, probability, information theory and control theory, and several other fields. We refer the reader to \cite{BS}, \cite{Dou}, \cite{RR}, and \cite{HKZ}  for a   comprehensive account on Toeplitz operators.

A polynomial $q\in \CC\left<Z_1,\ldots, Z_n\right>$ in $n$ noncommutative indeterminates is called {\it positive regular} if all its coefficients are positive, $q(0)=0$, and the coefficients of the linear terms $Z_1,\ldots, Z_n$ are different from zero. If $q=\sum_\alpha  a_\alpha Z_\alpha$ and $X=(X_1,\ldots, X_n)\in B(\cH)^n$, we define the completely positive map
$$
\Phi_{q, X}:B(\cH)\to B(\cH),\qquad \Phi_{q,X}(Y):=\sum_\alpha a_\alpha X_\alpha Y X_\alpha^* .
$$
For each $m\geq 1$,  we define the {\it noncommutative regular domain}
$$
{\bf D}_q^m(\cH):=\left\{ X:=(X_1,\ldots, X_n)\in B(\cH)^n: \ (id-\Phi_{q,X})^k(I)\geq 0 \  \text{\rm  for } 1\leq k\leq m\right\}.
$$
According to \cite{Po-domains} and  \cite{Po-domains-models}, each such a domain has a universal model $(W_1,\ldots, W_n)$ consisting of weighted left creation operators acting on the full Fock space with $n$ generators. We mention a few remarkable particular cases.

{\bf Single variable case: $n=1$}.
\begin{enumerate}
\item[(i)] If  $m=1$ and
$q=Z$,   the corresponding domain ${\bf D}_q^m(\cH)$ coincides with the closed unit ball  $[B(\cH)]_1:=\{X\in B(\cH): \ \|X\|\leq 1\}$, the study of which has generated the  Nagy-Foia\c s theory of contractions (see \cite{SzFBK-book}).
In this case, the universal model is the unilateral shift $S$  acting on the Hardy space $H^2(\DD)$. The Toeplitz operators on the Hardy space $H^2(\DD)$ have been studied extensively (see  for example \cite{Dou}, \cite{RR})
\item[(ii)] If
$m\geq 2$ and    $q=Z$,  the
    corresponding domain coincides with the set of all
    $m$-hypercontractions  studied by Agler  in \cite{Ag2},
    and  recently by
    Olofsson \cite{O1}, \cite{O2}. The corresponding universal model is the unilateral shift acting on the weighted Bergman space $A_m(\DD)$,  the Hilbert space of all analytic functions on the unit disc $\DD$ with
$$
\|f\|^2:=\frac{m-1}{\pi}\int_\DD|f(z)|^2 (1-|z|^2)^{m-2}dz<\infty.
$$
In \cite{LO}, Louhichi and Olofsson obtain a Brown-Halmos type characterization of Toeplitz operators with harmonic symbols on $A_m(\DD)$, which can be seen as  a  reproducing kernel Hilbert space with reproducing  kernel given by $\kappa_m(z,w):=(1-z\bar w)^{-m}$, $z,w\in \DD$. Their result was recently extended by Eschmeier and Langend\" orfer  \cite{EL} to the analytic functional  Hilbert space $H_m(\BB)$ on the unit ball $\BB\subset \CC^n$ given by the reproducing kernel $\kappa_m(z,w):=\left(1-\left<z,w\right>\right)^{-m}$ for  $z,w\in \BB$, where  $m\geq 1$.

    \end{enumerate}

{\bf Multivariable noncommutative case: $n\geq 2$}.
\begin{enumerate}
\item[(i)]
When $m=1$ and     $q=Z_1+\cdots +Z_n$,   the noncommutative domain ${\bf D}_q^m(\cH)$ coincides with the closed unit ball
$[B(\cH)^n]_1:=\{(X_1,\ldots, X_n):\ X_1 X_1^*+\cdots +X_nX_n^*\leq I\}$, the study of which has generated  a free analogue of Nagy-Foia\c s theory. The corresponding universal model is the $n$-tuple of left creation operators  $(S_1,\ldots, S_n)$ acting on the full Fock space with $n$ generators. A study of  unweighted multi-Toeplitz operators on the full Fock space with $n$ generators   was initiated in \cite{Po-multi}, \cite{Po-analytic} and has had an important impact in multivariable operator theory  and the structure of free semigroups algebras (see \cite{DP2}, \cite{DKP}, \cite{DLP}, \cite{Po-entropy}, \cite{Po-pluriharmonic}, \cite{Ken1}, \cite{Ken2}).

\item[(ii)]  When $m\geq 1$, $n\geq 1$, and $q$ is
   any positive regular  polynomial
     the domain ${\bf D}_q^m(\cH)$ was studied  in   \cite{Po-domains} (when $m=1$), and in \cite{Po-domains-models} (when $m\geq 2$). In this case, the corresponding universal model is an $n$-tuple of weighted left creation operators acting on the full Fock space.
We remark that, in the particular case when
  $m\geq 2$ and  $q=Z_1+\cdots +Z_n$,  the corresponding domain can be seen as a noncommutative $m$-hyperball,  the elements
    of  which  can be viewed  as
    multivariable noncommutative analogues of Agler's
    $m$-hypercontractions. As far as we know,  Toeplitz operators have not  been introduced or studied  in this very general setting.

\end{enumerate}
The goal of the present paper is to initiate the study of weighted multi-Toeplitz operators associated with noncommutative regular domains ${\bf D}_q^m(\cH)\subset  B(\cH)^n$, $m, n\geq 1$, when $q\in \CC\left<Z_1,\ldots, Z_n\right>$  is any positive regular polynomial in noncommutative indeterminates. This is accompanied by the study of their symbols which are  free pluriharmonic functions on the interior of the domain  ${\bf D}_q^m(\cH)$.

In Section 1, we present some background from  \cite{Po-domains}  and \cite{Po-domains-models} on the noncommutative domains ${\bf D}_q^m(\cH)$, their universal models, and the associated noncommutative Berezin transforms.

In Section 2, we introduce  the weighted multi-Toeplitz operators which are acting on  the full Fock space $F^2(H_n)$ with $n$ generators and are  associated with the noncommutative domain ${\bf D}_q^m(\cH)\subset B(\cH)^n$. We show that they are uniquely determined by their free pluriharmonic symbols
$$
\varphi(X_1,\ldots, X_n)=\sum_{k=1}^\infty \sum_{\alpha\in \FF_n^+, |\alpha|=k} b_\alpha X_\alpha^* +\sum_{k=0}^\infty \sum_{\alpha\in \FF_n^+, |\alpha|=k} a_\alpha X_\alpha, \qquad a_\alpha, b_\alpha\in \CC,
$$
where $\FF_n^+$ is the unital free semigroup with $n$ generators and the convergence of the series is in the operator norm topology for any $n$-tuple $(X_1,\ldots, X_n)$ in the interior of ${\bf D}_q^m(\cH)$. We prove that the set of all weighted multi-Toeplitz
operators coincides with
$$
\overline{\cA({\bf D}_q^m)^*+\cA({\bf D}_q^m)}^{\text\rm WOT},
$$
where the domain algebra $\cA({\bf D}_q^m)$ is the norm-closed unital non-selfadjoint algebra generated by the universal model $(W_1,\ldots, W_n)$  of the noncommutative domain ${\bf D}_q^m(\cH)$. In the particular case when $n=1$ and $q=Z$, we obtain a characterization of the Toeplitz operators with harmonic symbol on the Bergman space $A_m(\DD)$, which should be compared with the corresponding result from \cite{LO}.

In Section 3, we provide basic results concerning the free pluriharmonic functions on the  noncommutative domain ${\bf D}_q^m(\cH)^\circ$ and show that they are characterized by a mean value property. This result is used to obtain an analogue of Weierstrass theorem for free pluriharmonic functions and to show that the set $Har(({\bf D}_q^m)^\circ)$  of all pluriharmonic functions is a complete metric space with respect to an appropriate metric $\rho$. We also obtain, in this section,  a Schur type result \cite{Sc} characterizing the free pluriharmonic functions with positive real parts in terms of positive semi-definite weighted multi-Toeplitz kernels.

Section 4 concerns  the space $Har^\infty(({\bf D}_q^m)^\circ)$  of all bounded   free pluriharmonic functions on ${\bf D}_q^m(\cH)^\circ$. One of the main results states that
$F\in Har^\infty(({\bf D}_q^m)^\circ)$ if and only if it is  the noncommutative Berezin transform of a weighted multi-Toeplitz operator. Moreover, we prove that the map
$$
\Phi:Har^\infty(({\bf D}_{q}^m)^\circ)\to \overline{{\cA}({\bf D}_q^m)^*+{\cA}({\bf D}_q^m)}^{WOT}
$$
  defined by
 $\Phi(F):=\text{SOT-}\lim_{r\to 1} F(rW)$
 is a completely   isometric isomorphism of operator spaces. A noncommutative version of the Dirichlet extension problem for harmonic functions (see \cite{H}) is also provided. We prove that $F\in Har(({\bf D}_q^m)^\circ)$ has a continuous extension in the operator norm topology to ${\bf D}_q^m(\cH)$ if and only if there exists a multi-Toeplitz operator
  $\psi\in  \overline{{\cA}({\bf D}_q^m)^*+{\cA}({\bf D}_q^m)}^{\|\cdot\|}$ such that $F$ is the noncommutative
  Berezin transform of $\psi$.

In Section 5, using noncommutative Cauchy transforms associated with the domain ${\bf D}_q^m(\cH)$, we provide a free analytic functional calculus for $n$-tuples of operators $X=(X_1,\ldots, X_n)\in B(\cH)^n$ with the spectral radius of the reconstruction operator $R_{\tilde q, X}$ strictly less than 1. This extends to free pluriharmonic functions,  proving that the map
$$\Psi_{q,X}: \left(Har(({\bf D}_q^m)^\circ), \rho\right)\to  \left(B(\cH), \|\cdot \|\right)
$$ defined by $\Psi_{q,X}(G):=G(X)$ is continuous  and its restriction $\Psi_{q,X}|_{Hol(({\bf D}_q^m)^\circ)}$ is a continuous unital algebra homomorphism. Several consequences of this result are also provided.

We should mention that our results are presented in the more general setting of weighted multi-Toeplitz matrices with operator-valued entries and  free pluriharmonic functions  with operator-valued coefficients, while the noncommutative domain ${\bf D}_f^m(\cH)$ is generated by any  positive regular free holomorphic functions  $f$ in a neighborhood of the origin.

In a forthcoming paper \cite{Po-Toeplitz-Hyperball}, we obtain a Brown-Halmos characterization of the weighted  multi-Toeplitz operators associated with the noncommutative  $m$-hyperball (the case when $q=Z_1+\cdots + Z_n, m\geq 2$) which is a noncommutative version of Eschmeier and Langend\" orfer    recent commutative result \cite{EL}. This result shows that the weighted multi-Toeplitz are characterized by an  algebraic equation involving the universal model $(W_1,\ldots, W_n)$ of the noncommutative $m$-hyperball. It remains to be seen if this characterization  extends to the more general domains ${\bf D}_q^m$, where $q$ is
   any positive regular  polynomial.

\bigskip

\section{Noncommutative domains, universal models, and Berezin transforms}

This section contains some definitions and  the necessary background from  \cite{Po-domains}  and \cite{Po-domains-models} on the noncommutative regular domains ${\bf D}_f^m(\cH)$, their universal models, and the associated noncommutative Berezin transforms.

Let $\FF_n^+$ be the unital free semigroup on $n$ generators
$g_1,\ldots, g_n$ and the identity $g_0$.  The length of $\alpha\in
\FF_n^+$ is defined by $|\alpha|:=0$ if $\alpha=g_0$  and
$|\alpha|:=k$ if
 $\alpha=g_{i_1}\cdots g_{i_k}$, where $i_1,\ldots, i_k\in \{1,\ldots, n\}$.
 If    $Z_1,\ldots,Z_n$ are noncommutative
indeterminates, we denote $Z_\alpha:=
Z_{i_1}\ldots Z_{i_k}$ if $\alpha=g_{i_1}\ldots g_{i_k}\in \FF_n^+$,
\ $i_1,\ldots i_k\in \{1,\ldots,n\}$, and $Z_{g_0}:=1$.
Similarly, if  $X:=(X_1,\ldots, X_n)\in B(\cH)^n$, where $B(\cH)$ is the algebra
of all bounded linear operators on the Hilbert space $\cH$,    we
denote $X_\alpha:= X_{i_1}\cdots X_{i_k}$  and $X_{g_0}:=I_\cH$.
A formal power series   $f:= \sum_{\alpha\in \FF_n^+}
a_\alpha Z_\alpha$, \ $a_\alpha\in \CC$,  in noncommutative indeterminates $Z_1,\ldots, Z_n$,  is  called {\it free holomorphic
function} on the noncommutative ball  $[B(\cH)^n]_\rho$ for some
$\rho>0$, where
$$[B(\cH)^n]_\rho:=\{(X_1,\ldots, X_n)\in B(\cH)^n: \
\|X_1X_1^*+\cdots + X_nX_n^*\|^{1/2}< \rho\},
$$
if the series $\sum_{k=0}^\infty \sum_{|\alpha|=k} a_\alpha
X_\alpha$ is convergent in the operator norm topology for any
$(X_1,\ldots, X_n)\in [B(\cH)^n]_\rho$.
  According to  \cite{Po-holomorphic},  $f$  is  a free holomorphic
function on   $[B(\cH)^n]_\rho$ for any Hilbert space $\cH$ if and only if
\begin{equation*}
\limsup_{k\to\infty} \left( \sum_{|\alpha|=k}
|a_\alpha|^2\right)^{1/2k}\leq \frac{1}{\rho}.
\end{equation*}
Throughout this paper, we
  assume that  $a_\alpha\geq 0$ for any $\alpha\in \FF_n^+$, \ $a_{g_0}=0$,
 \ and  $a_{g_i}>0$ if  $i\in\{1,\ldots, n\}$.
 A function $f$ satisfying  all these conditions on the coefficients is
 called  {\it positive regular free holomorphic function on}
 $[B(\cH)^n]_\rho$.
 Let $\Phi_{f,X}: B(\cH)\to B(\cH)$ be the completely positive linear map  given by $\Phi_{f,X}(Y):=\sum_{|\alpha|\geq 1} a_\alpha X_\alpha Y X_\alpha^*$ for  $Y\in B(\cH)$, where the convergence is in the week operator topology, and define the {\it noncommutative regular domain}
$$
{\bf D}_f^m(\cH):=\left\{ X:=(X_1,\ldots, X_n)\in B(\cH)^n: \ (id-\Phi_{f,X})^k(I)\geq 0 \  \text{\rm  for } 1\leq k\leq m\right\}.
$$
  We saw in  \cite{Po-domains-models}, that $X\in {\bf D}_f^m(\cH)$ if and only if $\Phi_{f,X}(I)\leq I$ and $(id-\Phi_{f,X})^m(I)\geq 0$. The {\it abstract noncommutative domain} ${\bf D}_f^m$  is the disjoint union $\coprod_{\cH} {\bf D}_f^m(\cH)$, over all Hilbert spaces $\cH$.
We associate with  the abstract domain ${\bf D}_f^m$
   a unique $n$-tuple
$(W_1,\ldots, W_n)$ of weighted shifts, as follows.
 Define  $b_{g_0}^{(m)}:=1$ and
\begin{equation}
\label{b-al}
 b_\alpha^{(m)}= \sum_{j=1}^{|\alpha|}
\sum_{{\gamma_1\cdots \gamma_j=\alpha }\atop {|\gamma_1|\geq
1,\ldots, |\gamma_j|\geq 1}} a_{\gamma_1}\cdots a_{\gamma_j}
\left(\begin{matrix} j+m-1\\m-1
\end{matrix}\right)  \qquad
\text{ if } \ \alpha\in \FF_n^+, |\alpha|\geq 1.
\end{equation}
Let $H_n$ be an $n$-dimensional complex  Hilbert space with orthonormal
      basis
      $e_1$, $e_2$, $\dots,e_n$, where $n\in\NN:=\{1,2,\dots\}$.        We consider
      the full Fock space  of $H_n$ defined by
      $$F^2(H_n):=\bigoplus_{k\geq 0} H_n^{\otimes k},$$
      where $H_n^{\otimes 0}:=\CC 1$ and $H_n^{\otimes k}$ is the Hilbert
      tensor product of $k$ copies of $H_n$.
Let  $D_i:F^2(H_n)\to F^2(H_n)$, $i\in \{1,\ldots, n\}$, be the  diagonal
operators defined  by setting
$$
D_ie_\alpha:=\sqrt{\frac{b_\alpha^{(m)}}{b_{g_i \alpha}^{(m)}}}
e_\alpha,\qquad
 \alpha\in \FF_n^+,
$$
where $\{e_\alpha\}_{\alpha\in \FF_n^+}$ is the orthonormal basis of the full Fock space $F^2(H_n)$.
The {\it weighted left creation  operators}
$W_i:F^2(H_n)\to F^2(H_n)$ associated with ${\bf D}_f^m$ are defined  by $W_i:=S_iD_i$, where
 $S_1,\ldots, S_n$ are the left creation operators on the full
 Fock space $F^2(H_n)$, i.e.
      $$
       S_i\varphi:=e_i\otimes\varphi, \qquad  \varphi\in F^2(H_n),\  i\in \{1,\ldots,n\}.
      $$
       A simple calculation reveals that
\begin{equation}\label{WbWb}
W_\beta e_\gamma= \frac {\sqrt{b_\gamma^{(m)}}}{\sqrt{b_{\beta
\gamma}^{(m)}}} e_{\beta \gamma} \quad \text{ and }\quad W_\beta^*
e_\alpha =\begin{cases} \frac
{\sqrt{b_\gamma^{(m)}}}{\sqrt{b_{\alpha}^{(m)}}}e_\gamma& \text{ if
}
\alpha=\beta\gamma \\
0& \text{ otherwise }
\end{cases}
\end{equation}
 for any $\alpha, \beta \in \FF_n^+$.
      We recall  from \cite{Po-domains-models} that the weighted left creation
  operators $W_1,\ldots, W_n$      have the following properties:
 \begin{enumerate}
 \item[(i)] $\sum\limits_{|\beta|\geq 1} a_\beta W_\beta W_\beta^*\leq I$, where the
convergence is in the strong operator topology;
 \item[(ii)]
 $\left(id-\Phi_{f,W}\right)^{m}(I)=P_\CC$, where $P_\CC$ is the
 orthogonal projection of $F^2(H_n)$ on $\CC 1\subset F^2(H_n)$, and the map
 $\Phi_{f,W}:B(F^2(H_n))\to
B(F^2(H_n))$    is defined  by
$$\Phi_{f,W}(Y):=\sum\limits_{|\alpha|\geq 1} a_\alpha W_\alpha
YW_\alpha^*,
$$
where the convergence is in the weak operator topology;
 \item[(iii)]  $W:=(W_1,\ldots, W_n)$  is a {\it pure} element of the   domain ${\bf D}_f^m(F^2(H_n))$, i.e. $\lim\limits_{p\to\infty} \Phi^p_{f,W}(I)=0$ in the strong operator
 topology.
 \end{enumerate}

The right creation operators are defined by $
       R_i\varphi:=\varphi\otimes e_i$,   $i\in \{1,\ldots,n\}$.
We can also define the {\it weighted right creation operators}
$\Lambda_i:F^2(H_n)\to F^2(H_n)$ by setting $\Lambda_i:= R_i G_i$,
$i=1,\ldots, n$,  where
 each  diagonal operator $G_i$, $i=1,\ldots,n$,  is given by
$$
G_ie_\alpha:=\sqrt{\frac{b_\alpha^{(m)}}{b_{ \alpha g_i}^{(m)}}}
e_\alpha,\quad
 \alpha\in \FF_n^+,
$$
where the coefficients $b_\alpha^{(m)}$, $\alpha\in \FF_n^+$, are
described  by relation \eqref{b-al}. In this case, we have
\begin{equation}\label{WbWb-r}
\Lambda_\beta e_\gamma= \frac {\sqrt{b_\gamma^{(m)}}}{\sqrt{b_{
\gamma \tilde\beta}^{(m)}}} e_{ \gamma \tilde \beta} \quad \text{
and }\quad \Lambda_\beta^* e_\alpha =\begin{cases} \frac
{\sqrt{b_\gamma^{(m)}}}{\sqrt{b_{\alpha}^{(m)}}}e_\gamma& \text{ if
}
\alpha=\gamma \tilde \beta \\
0& \text{ otherwise }
\end{cases}
\end{equation}
 for any $\alpha, \beta \in \FF_n^+$, where $\tilde \beta$ denotes
 the reverse of $\beta=g_{i_1}\cdots g_{i_k}$, i.e.,
 $\tilde \beta=g_{i_k}\cdots g_{i_1}$.
As in the case of weighted left creation operators, one can show
that
 \begin{equation}
 \label{tild-Lamb} \sum\limits_{|\beta|\geq 1}
a_{\tilde\beta} \Lambda_\beta
 \Lambda_\beta^*\leq I\quad \text{ and } \quad \left(id-\Phi_{\tilde
 f,\Lambda}\right)^m(I)=P_\CC,
 \end{equation}
 where
  $\tilde{f}(Z):=\sum_{|\alpha|\geq 1} a_{\tilde \alpha} Z_\alpha$, $\tilde
\alpha$ denotes the reverse of $\alpha$, and $\Phi_{\tilde f,
\Lambda}(Y):=\sum_{|\alpha|\geq 1} a_{\tilde \alpha} \Lambda_\alpha
Y \Lambda_\alpha^*$ for any  $Y\in B(F^2(H_n))$, with the convergence is in
the weak operator topology.

Let $X:=(X_1,\ldots, X_n)\in {\bf D}_f^m(\cH)$ and let
$K_{f,X}^{(m)}:\cH\to F^2(H_n)\otimes
\overline{\Delta_{m,X}(\cH)}$  be the {\it noncommutative Berezin kernel} defined by
\begin{equation*}
 K_{f,X}^{(m)}h:=\sum_{\alpha\in \FF_n^+} \sqrt{b_\alpha^{(m)}}
e_\alpha\otimes \Delta_{m,X} X_\alpha^* h,\qquad h\in \cH,
\end{equation*}
where $\Delta_{m,X}:=\left[(I- \Phi_{f,X})^m(I) \right]^{1/2}$ and the coefficients $b_{\alpha}^{(m)}$ are given by
relation \eqref{b-al}.
We know that
\begin{equation*}
K_{f,X}^{(m)}X_i^*=(W_i^*\otimes
I )K_{f,X}^{(m)} \qquad i\in \{1,\ldots, n\}.
\end{equation*}
 Assume that $X$ is a {\it pure} $n$-tuple, i.e. $\Phi_{f,X}^k(I)\to 0$ strongly, as $k\to \infty$. Then $K_{f,X}^{(m)}$ is an isometry and     the $n$-tuple   $W:=(W_1,\ldots, W_n)$
 plays the role of the {\it universal model for the  noncommutative domain}
${\bf D}_f^m$.

Let $\varphi(W):=\sum\limits_{\beta\in \FF_n^+} c_\beta
W_\beta$, $c_\beta \in \CC$, be a formal sum with the property  that $\sum_{\beta\in
\FF_n^+} |c_\beta|^2 \frac{1}{b_\beta^{(m)}}<\infty$.
In \cite{Po-domains-models}, we proved that
$\sum\limits_{\beta\in \FF_n^+} c_\beta W_\beta (p)\in F^2(H_n)$ for
any $p\in \cP$, where $\cP\subset F^2(H_n)$ is the set of all polynomial in
$e_\alpha$, $\alpha\in \FF_n^+$.
If
$$
\sup_{p\in\cP, \|p\|\leq 1} \left\|\sum\limits_{\beta\in \FF_n^+}
c_\beta W_\beta (p)\right\|<\infty,
$$
then there is a unique bounded operator acting on $F^2(H_n)$, which
we should also  denote by $\varphi(W)$, such that
$$
\varphi(W)p=\sum\limits_{\beta\in \FF_n^+} c_\beta
W_\beta (p)\quad \text{ for any } \ p\in \cP.
$$
The set of all operators $\varphi(W)\in B(F^2(H_n))$
satisfying the above-mentioned properties is denoted by
$F^\infty({\bf D}^m_f)$.   One can prove that $F^\infty({\bf D}^m_f)$
is a Banach algebra, which we call Hardy algebra associated with the
noncommutative domain ${\bf D}^m_f$.
We introduce the  domain  algebra $\cA({\bf D}^m_f)$  to be the norm closure
of all polynomials in the weighted left creation operators
$W_1,\ldots, W_n$ and the identity. Using the weighted right
creation operators
 associated with ${\bf D}^m_f$, one can  also  define    the corresponding
  domain algebra ${\cR}({\bf D}^m_f)$.

 In a similar manner,  using the weighted right creation
operators $\Lambda:=(\Lambda_1,\ldots, \Lambda_n)$
 associated with ${\bf D}^m_f$, one can   define    the corresponding
     the Hardy algebra $R^\infty({\bf D}^m_f)$.
More precisely, if $g(\Lambda)=\sum\limits_{\beta\in \FF_n^+} c_{\tilde\beta
}\Lambda_\beta $ is a formal sum with the property  that
$\sum_{\beta\in \FF_n^+} |c_\beta|^2
\frac{1}{b_\beta^{(m)}}<\infty$  and such that
$$
\sup_{p\in\cP, \|p\|\leq 1} \left\|\sum\limits_{\beta\in \FF_n^+}
 c_{\tilde\beta} \Lambda_\beta (p)\right\|<\infty,
$$
then there is a unique bounded operator on $F^2(H_n)$, which we also
 denote by $g(\Lambda)$, such that
$$
g(\Lambda)p=\sum\limits_{\beta\in \FF_n^+}
c_{\tilde\beta} \Lambda_\beta (p)\quad \text{ for any } \ p\in \cP.
$$
The set of all operators $g(\Lambda)\in
B(F^2(H_n))$
satisfying the above-mentioned properties is denoted by $R^\infty({\bf D}^m_f)$. We proved in \cite{Po-domains-models} that $F^\infty({\bf D}^m_f)'= R^\infty({\bf D}^m_f)$ and $F^\infty({\bf D}^m_f)''=F^\infty({\bf D}^m_f)$, where $'$ stands for the commutant.

The  {\it noncommutative Berezin transform at} $X\in {\bf D}_f^m(\cH)$, where $ X$ is a pure element,
 is the map ${\bf B}^{(m)}_X: B(F^2(H_{n}))\to B(\cH)$
 defined by
 \begin{equation*}
 {\bf B}^{(m)}_X[g]:= {K_{f,X}^{(m)}}^* (g\otimes I_\cH)K_{f,X}^{(m)},
 \quad g\in B(F^2(H_n)),
 \end{equation*}
where the  $K_{f,X}^{(m)}:\cH \to F^2(H_n)\otimes
\cH$  is noncommutative Berezin kernel.
 Let $\cP({W})$  be the set of all polynomials $p({W})$  in  the operators ${ W}_{i}$, $i\in \{1,\ldots, k\}$,  and the identity.
  If $g$ is in the operator space
  $$\cS:=\overline{\text{\rm  span}} \{ p({W})q({W})^*:\
p({W}),q({W}) \in  \cP({W})\},
$$
where the closure is in the operator norm, we  define the Berezin transform at  $X\in {\bf D}_f^m(\cH)$, by
  $${\bf B}^{(m)}_X[g]:=\lim_{r\to 1} {K_{f,rX}^{(m)}}^* (g\otimes I_\cH)K_{f,rX}^{(m)},
 \qquad g\in  \cS,
 $$
 where the limit is in the operator norm topology.
In this case, the Berezin transform at ${X}$ is a unital  completely positive linear  map such that
 $${\bf B}^{(m)}_X({ W}_{\alpha} {W}_{\beta}^*)={X}_{\alpha} {X}_{\beta}^*, \qquad \alpha, \beta \in \FF_{n}.
 $$
   If, in addition,
 ${X}$ is a  pure $n$-tuple in ${\bf D}_f^m(\cH)$,
then
$\lim_{r\to 1} {\bf B}^{(m)}_{rX}[g]= {\bf B}^{(m)}_{X}[g]$, $g\in \cS.$
More on  noncommutative Berezin transforms  and their applications can be found in \cite{Po-poisson}, \cite{Po-domains-models},  \cite{Po-domains}, \cite{Po-Berezin2}, and  \cite{Po-Berezin1}.

\bigskip

\section{Wieghted multi-Toeplitz operators on Fock spaces}

In this section,  we introduce  the weighted multi-Toeplitz operators     associated with the noncommutative domain ${\bf D}_f^m(\cH)\subset B(\cH)^n$. We show that they are uniquely determined by their free pluriharmonic symbols and provide a characterization in terms of the domain algebra $\cA({\bf D}_f^m)$.

In what follows, we need some notation.
 If $\omega, \gamma\in \FF_n^+$,
we say that $\omega
\geq_{r}\gamma$ if there is $\sigma\in
\FF_n^+ $ such that $\omega=\sigma \gamma$. In this
case  we set $\omega\backslash_r \gamma:=\sigma$. If $\sigma\neq g_0$ we write $\omega>_r \gamma$. We say that $\omega$ and $\gamma$ are {\it comparable} if either $\omega
\geq_{r}\gamma$ or $\gamma>_r\omega$.
 If $\omega, \gamma\in \FF_n^+$ are comparable, we consider the weights
 $$
 \lambda_{\omega, \gamma}:=\begin{cases} \sqrt{\frac{b_\omega^{(m)}}{b_\gamma^{(m)}}},&\ \text{ if } \ \omega\geq_r\gamma,\\
 \sqrt{\frac{b_\gamma^{(m)}}{b_\omega^{(m)}}},&\ \text{ if } \ \gamma>_r\omega,
 \end{cases}
 $$
 where the coefficients $b_\alpha^{(m)}$, $\alpha\in \FF_n^+$,  are given by relation \eqref{b-al}. Let  $\cE$ be a separable Hilbert space and let $\left[ C_{\omega,\gamma}\right]_{\FF_n^+\times \FF_n^+}$  be
   the  operator matrix representation of $T\in B(\cE\otimes F^2(H_n))$, i.e.
 $$
 \left<C_{\omega,\gamma} x,y\right>:=\left<T(x\otimes e_\gamma), y\otimes e_\omega\right>
 $$
 for any $\omega, \gamma\in \FF_n^+$ and $x,y\in \cE$.

 \begin{definition}\label{def-Toeplitz}
  We say that $T$ is a weighted right multi-Toeplitz operator    if
 for each $i\in \{1,\ldots, n\}$ and $\omega, \gamma, \alpha, \beta\in \FF_n^+$,
\begin{equation*}
\lambda_{\omega g_i,\gamma g_i}C_{\omega g_i,\gamma g_i}=\lambda_{\omega, \gamma} C_{\omega,\gamma},\qquad \text{ if } \ \omega, \gamma \ \text{  are comparable},
\end{equation*}
 and $C_{\alpha, \beta}=0$ if $\alpha, \beta$ are not comparable.
\end{definition}
We remark that when $n=m=1$, $f=Z$, and $\cE=\CC$ we recover   the classical Toeplitz operators on the
Hardy space $H^2(\DD)$. Also if $n\geq 2$, $m=1$,  and $f=Z_1+\cdots +Z_n$ we obtain the unweighted  right multi-Toeplitz operators on the full Fock space $F^2(H_n)$ (see \cite{Po-multi}, \cite{Po-analytic} and \cite{Po-pluriharmonic}). In this case, we have $b_\alpha^{(m)}=1$ for any $\alpha\in \FF_n^+$  and the condition above becomes
$$
C_{\omega g_i, \gamma g_i}=
\begin{cases} C_{\omega, \gamma},& \text{if } \ \omega\geq_r\gamma \ \text{or } \ \gamma>_r\omega,\\
0,& \text{otherwise},
\end{cases}
$$
and $C_{\alpha, \beta}=0$ if $\alpha, \beta$ are not comparable.

For an equivalent and more transparent  definition of  weighted  right multi-Toeplitz operators on the full Fock space $F^2(H_n)$ see the remarks following the next theorem.

\begin{theorem}\label{Fourier}
Any  weighted right multi-Toeplitz operator $T\in B(\cE\otimes F^2(H_n))$ has a    formal Fourier  representation
$$
\varphi(W):= \sum_{|\alpha|\geq 1} B_{(\alpha)}\otimes  W_\alpha^*
+ A_{(0)}\otimes  I +\sum_{|\alpha|\geq 1} A_{(\alpha)}\otimes W_\alpha,
$$
 where $\{A_{(\alpha)}\}_{\alpha\in \FF_n^+}$ and $\{B_{(\alpha)}\}_{\alpha\in
\FF_n^+\backslash \{g_0\}}$  are some operators on the Hilbert space $\cE$,  such that  $$
Tq=\varphi(W)q,\qquad
q=\sum_{|\alpha|\leq k} h_\alpha\otimes e_\alpha,
$$
for any  $h_\alpha\in \cE$ and $k\in \NN$.
 If $T_1$,  $T_2$ are weighted right multi-Toeplitz operators  having the
same formal Fourier  representation, then $T_1=T_2$.
\end{theorem}
\begin{proof}  First, we note that, using Definition  \ref{def-Toeplitz}, one can prove that $T\in B(\cE\otimes F^2(H_n))$ is a weighted right multi-Toeplitz operator if and only if  the entries of its   matrix representation $\left[ C_{\omega,\gamma}\right]_{\FF_n^+\times \FF_n^+}$  satisfy the following relations:
\begin{enumerate}
\item[(i)] $C_{\sigma\gamma, \gamma}=\sqrt{\frac{b_\sigma^{(m)} b_\gamma^{(m)}}{b_{\sigma\gamma}^{(m)}}}C_{\sigma, g_0}$ for any $\sigma, \gamma\in \FF_n^+$;
    \item[(ii)] $C_{\gamma, \sigma\gamma}=\sqrt{\frac{b_\sigma^{(m)} b_\gamma^{(m)}}{b_{\sigma\gamma}^{(m)}}}C_{g_0,\sigma}$ for any $\sigma, \gamma\in \FF_n^+$;
        \item[(iii)] $C_{\alpha, \beta}=0$ if $(\alpha, \beta) \in \FF_n^+\times \FF_n^+$ is not of the form  $(\sigma\gamma, \gamma)$ or $(\gamma, \sigma\gamma)$ for   $\sigma, \gamma\in \FF_n^+$.
         \end{enumerate}
Consequently,  $T\in B(\cE\otimes F^2(H_n))$ is a
weighted right multi-Toeplitz   if and only if

\begin{equation}\label{MT}
\left<T(x\otimes e_\gamma), y\otimes e_\omega\right>
=
\begin{cases}\frac{\sqrt{b_{\omega\backslash_r\gamma}^{(m)}}\sqrt{b_\gamma^{(m)}}}
{\sqrt{b_\omega^{(m)}}}\left<T(x\otimes 1),y\otimes e_{\omega\backslash_r\gamma}\right>, & \text{if } \ \omega\geq_r\gamma,\\
\frac{\sqrt{b_{\gamma\backslash_r\omega}^{(m)}}\sqrt{b_\omega^{(m)}}}
{\sqrt{b_\gamma^{(m)}}}\left<T(x\otimes e_{\gamma\backslash_r\omega}),y\otimes 1\right>, & \text{if } \ \gamma >_r\omega,\\
0, & \text{otherwise}.
\end{cases}
\end{equation}
We define the formal  Fourier representation of
$T$  by setting
$$
\varphi(W):= \sum_{|\alpha|\geq 1} B_{(\alpha)}\otimes  W_\alpha^*
+ A_{(0)}\otimes  I +\sum_{|\alpha|\geq 1} A_{(\alpha)}\otimes W_\alpha,
$$
where the coefficients are given by
\begin{equation}\label{f-coef}
\begin{split}
\left<A_{(\alpha)}x,y\right> &:=\sqrt{b_\alpha^{(m)}}\left< T(x\otimes 1), y\otimes
e_\alpha\right>, \quad \alpha\in \FF_n^+,\\
 \left<B_{(\alpha)}x,
y\right>&:=\sqrt{b_\alpha^{(m)}}\left< T(x\otimes e_\alpha), y\otimes 1\right>, \quad
\alpha\in \FF_n^+\backslash\{g_0\},
\end{split}
\end{equation}
for any $x,y\in \cE$. We also set $A_{(0)}:=A_{(g_0)}$.
Hence, we deduce that
$$
T(x\otimes 1)=\sum_{\alpha\in \FF_n^+}\frac{1}{\sqrt{b_\alpha^{(m)}}} A_{(\alpha)}x\otimes e_\alpha
$$
and
$$
T^*(x\otimes 1)=\sum_{\alpha\in \FF_n^+, |\alpha|\geq 1}\frac{1}{\sqrt{b_\alpha^{(m)}}} B_{(\alpha)}^*x\otimes e_\alpha
$$
for any $x\in \cE$. As a consequence, we can see that $\sum_{|\alpha|\geq 1} \frac{1}{b_\alpha^{(m)}}A_{(\alpha)}^* A_{(\alpha)}$ and
$\sum_{|\alpha|\geq 1} \frac{1}{b_\alpha^{(m)}} B_{(\alpha)} B_{(\alpha)}^*$ are WOT
convergent series. We note that
$$
\varphi(W)(x\otimes e_\beta):= \sum_{|\alpha|\geq 1} (B_{(\alpha)}\otimes  W_\alpha^*)(x\otimes e_\beta)
+  \sum_{\alpha\in \FF_n^+} (A_{(\alpha)}\otimes W_\alpha)(x\otimes e_\beta),
$$
is well-defined as a vector in $\cE\otimes F^2(H_n)$. Indeed, the first sum consists of finitely many terms, while the second one is equal to
$\sum_{\alpha\in \FF_n^+} A_{(\alpha)}x\otimes \sqrt{\frac{b_\beta^{(m)}}{b_{\alpha\beta}^{(m)}}}e_{\alpha\beta}$.
Using the definition of the coefficients $b_\alpha^{(m)}$, one can easily see that
$b_\alpha^{(m)}b_\beta^{(m)}\leq \left(\begin{matrix} |\beta|+m-1\\ m-1\end{matrix}\right) b_{\alpha\beta}^{(m)}$. This implies
$$
\sum_{\alpha\in \FF_n^+} \|A_{(\alpha)}x\|^2 \frac{b_\beta^{(m)}}{b_{\alpha\beta}^{(m)}}\leq \left(\begin{matrix} |\beta|+m-1\\ m-1\end{matrix}\right)\sum_{\alpha\in \FF_n^+} \|A_{(\alpha)}x\|^2 \frac{1}{b_{\alpha}^{(m)}}<\infty.
$$
Since $T$ is a weighted right multi-Toeplitz operator, we can use relations \eqref{MT} and  \eqref{f-coef}, to  obtain
\begin{equation}\label{TAB}
\left<T(x\otimes e_\gamma), y\otimes e_\omega\right>
=
\begin{cases}\frac{ \sqrt{b_\gamma^{(m)}}}
{\sqrt{b_\omega^{(m)}}}\left<A_{\omega\backslash_r\gamma}x,y\right>, & \text{if } \ \omega\geq_r\gamma,\\
\frac{ \sqrt{b_\omega^{(m)}}}
{\sqrt{b_\gamma^{(m)}}}\left<B_{\gamma\backslash_r\omega}x,y\right>, & \text{if } \ \gamma >_r\omega,\\
0, & \text{otherwise}.
\end{cases}
\end{equation}
Now, note that
$$
\left<\varphi(W)(x\otimes e_\gamma), y\otimes e_\omega\right>= \sum_{|\alpha|\geq 1} \left<B_{(\alpha)}x,y\right>\left< W_\alpha^* e_\gamma,e_\omega\right>
+  \sum_{\alpha\in \FF_n^+} \left<A_{(\alpha)}x, y\right>\left< W_\alpha e_\gamma, e_\omega\right>.
$$
Due to the definition of the weighted left creation operators $W_1,\ldots, W_n$, we have
$$
\left< W_\alpha e_\gamma, e_\omega\right>=\begin{cases}
\frac{ \sqrt{b_\gamma^{(m)}}}
{\sqrt{b_{\alpha\gamma}^{(m)}}},& if \ \omega=\alpha\gamma,\\
0,& otherwise.
\end{cases}
$$
for any $\alpha\in \FF_n^+$,
and
$$
\left< W_\alpha^* e_\gamma, e_\omega\right>=\begin{cases}
\frac{ \sqrt{b_\omega^{(m)}}}
{\sqrt{b_{\alpha\omega}^{(m)}}},& if \ \gamma=\alpha\omega,\\
0,& otherwise
\end{cases}
$$
for any $\alpha\in \FF_n^+$ with $|\alpha|\geq 1$.
Using these relations, we deduce that
\begin{equation}
\label{fi}
\left<\varphi(W)(x\otimes e_\gamma), y\otimes e_\omega\right>=
\begin{cases}
\left<\frac{ \sqrt{b_\gamma^{(m)}}}
{\sqrt{b_{\alpha\gamma}^{(m)}}}A_{(\alpha)} x,y\right>,& if \ \omega=\alpha\gamma, \alpha\in \FF_n^+,
\\
\left<\frac{ \sqrt{b_\omega^{(m)}}}
{\sqrt{b_{\alpha\omega}^{(m)}}}B_{(\alpha)} x,y\right>,& if \ \gamma=\alpha\omega, \alpha\in \FF_n^+ \ \text{with} \ |\alpha|\geq 1,
\\
0,& otherwise.
\end{cases}
\end{equation}
Comparing these relations with \eqref{TAB}, we conclude that
$$
\left<T(x\otimes e_\gamma), y\otimes e_\omega\right>
=\left<\varphi(W)(x\otimes e_\gamma), y\otimes e_\omega\right>
$$
for any $x,y\in \cE$ and $\gamma, \omega\in \FF_n^+$.
Consequently, we obtain  $T(x\otimes e_\gamma)=\varphi(W)(x\otimes e_\gamma)$. The last part of the theorem is now straightforward. The proof is complete.
\end{proof}

Let $F_{f,m}^2$ be the Hilbert space of formal power series in noncommutative indeterminates $Z_1,\ldots, Z_n$ with complete orthogonal basis $\{Z_\alpha: \ \alpha \in \FF_n^+\}$ and such that $\|Z_\alpha\|_{f,m}:=\frac{1}{\sqrt{b_\alpha^{(m)}}}$.  It is clear that 
$$
F_{f,m}^2=\left\{ \varphi:=\sum_{\alpha\in \FF_n^+} a_\alpha Z_\alpha: \ a_\alpha\in \CC \ \text{\rm and }\  \|\varphi\|_{f,m}^2:=
\sum_{\alpha\in \FF_n^+} \frac{1}{b_\alpha^{(m)}} |a_\alpha|^2<\infty\right\}.
$$
The left multiplication operators $L_1 ,\ldots, L_n $ are defined by
$L_i \xi:=Z_i\xi$ \, for all $\xi\in F^2_{f,m}$.
 Note that the operator 
$U_{f,m}:F^2(H_n)\to F^2_{f,m}$  defined by
$
U_{f,m}(e_\alpha):=\sqrt{b_\alpha^{(m)}} Z_\alpha$, $ \alpha\in \FF_n^+,
$
is unitary and 
   $U_{f,m}W_i^{(m)}=L_i^{(m)} U_{f,m}$  for any $ i\in \{1,\ldots, n\}.
$
A straightforward calculation reveals that $T\in B(\cE\otimes F^2(H_n))$  is  a  weighted right multi-Toeplitz operator if and only if $A:=U_{f,m}TU^*_{f,m}$  satisfies the condition
\begin{equation*} 
\left<A(x\otimes e_\gamma), y\otimes e_\omega\right>
=
\begin{cases}\frac{1}
{\sqrt{b_\omega^{(m)}}}\left<A_{\omega\backslash_r\gamma}x,y\right>, & \text{if } \ \omega\geq_r\gamma,\\
\frac{1}
{\sqrt{b_\gamma^{(m)}}}\left<B_{\gamma\backslash_r\omega}x,y\right>, & \text{if } \ \gamma >_r\omega,\\
0, & \text{otherwise},
\end{cases}
\end{equation*}
for some operators $\{A_{(\alpha)}\}_{\alpha\in \FF_n^+}$ and $\{B_{(\alpha)}\}_{\alpha\in
\FF_n^+\backslash \{g_0\}}$ in $B(\cH)$. Note that the Hilbert space $F_{f,m}^2$ can be seen as a weighted Fock space. In the particular case when $n=1$ and  $q=Z$,  it coincides with the weighted Bergman space $A_m(\DD)$, while $A$ is a  Toeplitz operator with operator-valued bounded harmonic symbol (see  \cite{LO} for the scalar case when $\cE=\CC$). All the results of the present paper can be written in the setting of multi-Toeplitz operators on weighted Fock spaces. However, we preferred this time   to put the weights on the left creation operators instead on the full Fock space.

We denote by $\boldsymbol{\cA}_\cE({\bf D}_f^m)$  the spatial tensor product
$B(\cE)\otimes_{min}\cA({\bf D}_f^m)$, where $\cA({\bf D}_f^m)$ is the noncommutative domain
algebra.  Let $\cP\subset F^2(H_n)$ be the set of all polynomials in $e_\alpha$, $\alpha\in \FF_n^+$.

The main result of this section is the following
characterization of the weighted right multi-Toeplitz operators in terms of their
Fourier representations, which can be viewed as their symbols.

\begin{theorem}\label{Toeplitz}
Let $\{A_{(\alpha)}\}_{\alpha\in \FF_n^+}$ and $\{B_{(\alpha)}\}_{\alpha\in
\FF_n^+\backslash \{g_0\}}$ be two sequences of  operators on a Hilbert space $\cE$.
Then
$$
\varphi(W):=\sum_{|\alpha|\geq 1} B_{(\alpha)}\otimes W_\alpha^*
+ A_{(0)}\otimes  I +\sum_{|\alpha|\geq 1} A_{(\alpha)}\otimes  W_\alpha
$$
is the formal Fourier representation of a weighted right multi-Toeplitz operator $T\in
B(\cE\otimes F^2(H_n))$ if and only if
\begin{enumerate}
\item[(i)]
$\sum_{|\alpha|\geq 1} \frac{1}{b_\alpha^{(m)}}A_{(\alpha)}^* A_{(\alpha)}$ and
$\sum_{|\alpha|\geq 1} \frac{1}{b_\alpha^{(m)}} B_{(\alpha)} B_{(\alpha)}^*$ are WOT
convergent series, and
\item[(ii)]
$\sup\limits_{0\leq r<1} \|\varphi(rW)\|<\infty$.
\end{enumerate}
Moreover, in this case,
\begin{enumerate}
\item[(a)]  for each $r\in [0,1)$, the operator
$$\varphi(rW):=\sum_{k=1}^\infty \sum_{|\alpha|=k}
B_{(\alpha)}\otimes  r^{|\alpha|} W_\alpha^* + A_{(0)}\otimes I
 +\sum_{k=1}^\infty
\sum_{|\alpha|=k} A_{(\alpha)}\otimes r^{|\alpha|} W_\alpha $$ is in the operator space
$\boldsymbol{\cA}_\cE({\bf D}_f^m)^* +\boldsymbol{\cA}_\cE({\bf D}_f^m) $, where the series are convergent in the operator
norm topology;
\item[(b)]
$T=\text{\rm SOT-}\lim\limits_{r\to 1} \varphi(rW)$,
and
\item[(c)]
$ \|T\|=\sup\limits_{0\leq r<1} \|\varphi(rW)\|=\lim\limits_{r\to 1} \|\varphi(rW)\|=\sup\limits_{q\in \cE\otimes\cP, \|q\|\leq 1}
\|\varphi(W)q\|.$
\end{enumerate}
\end{theorem}

\begin{proof} Assume that $T\in B(\cE\otimes F^2(H_n))$ is a weighted right multi-Toeplitz operator and that $\varphi(W)$ is its formal Fourier representation. Note that part (i) was proved in the proof of Theorem \ref{Fourier}.

Using the definition of the weighted left creation operators  and the fact that
$$b_\alpha^{(m)}b_\beta^{(m)}\leq \left(\begin{matrix} |\beta|+m-1\\ m-1\end{matrix}\right) b_{\alpha\beta}^{(m)},
$$
 we deduce that
$$
\|W_\alpha\|\leq \frac{1}{\sqrt{b_\alpha^{(m)}}} \left(\begin{matrix} |\alpha|+m-1\\ m-1\end{matrix}\right)^{1/2}.
$$
Since  the operators $W_\alpha, \alpha\in \FF_n^+, |\alpha|=k$ have orthogonal  ranges, we deduce that
$$
\left\|\sum_{\alpha\in \FF_n^+, |\alpha|=k} b_\alpha^{(m)}W_\alpha W_\alpha^*\right\|\leq \left(\begin{matrix} k+m-1\\ m-1\end{matrix}\right).
$$
Consequently
$$
\left\|\sum_{|\alpha|=k} A_{(\alpha)}\otimes r^{|\alpha|} W_\alpha\right\|=r^k \left\|\sum_{|\alpha|=k} \frac{1}{b_\alpha^{(m)}}A_{(\alpha)}^* A_{(\alpha)}\right\|^{1/2}\left(\begin{matrix} k+m-1\\ m-1\end{matrix}\right)^{1/2}
$$
which implies the convergence of the series
$\sum_{k=0}^\infty\left\|\sum_{|\alpha|=k} A_{(\alpha)}\otimes r^{|\alpha|} W_\alpha\right\|$.
 A similar result holds for the operators $B_{(\alpha)}$. Using part (i), we can easily see that $\varphi(rW)$ is in $\boldsymbol{\cA}_\cE({\bf D}_f^m)^* +\boldsymbol{\cA}_\cE({\bf D}_f^m) $, where the series in the definition of $\varphi(rW)$ are convergent in the operator
norm topology. This shows that part $(a)$ holds.
Now, we prove that, for any $r\in [0,1)$,
\begin{equation}
\label{KTK}
(I_\cE\otimes K_{f,rW}^{(m)})^*(T\otimes I_{F^2(H_n)})(I_\cE\otimes K_{f,rW}^{(m)})=\varphi(rW),
\end{equation}
where
$K_{f,rW}^{(m)}:F^2(H_n)\to F^2(H_n)\otimes \cD_{rW}$ is the noncommutative Berezin kernel defined by
$$
K_{f,rW}^{(m)}\xi=\sum_{\beta\in \FF_n^+}e_\beta\otimes \Delta_{m,rW}^{1/2} W_\beta^*\xi,\qquad \xi\in F^2(H_n),
$$
and $\cD_{rW}:=\overline{\Delta_{m,rW}^{1/2}(F^2(H_n))}$. Let $\gamma, \omega\in \FF_n^+$, set $q:=\max\{|\gamma|, |\omega|\}$ and define
$$
\varphi_q(W):=\sum_{1\leq|\alpha|\leq q} B_{(\alpha)}\otimes W_\alpha^*
+  \sum_{\alpha\in \FF_n^+,|\alpha|\leq q} A_{(\alpha)}\otimes  W_\alpha.
$$
Note that $W_\alpha^*(e_\gamma)=0$ if $|\alpha|>q$ and, similarly, $W_\beta^*(e_\omega)=0$ if $|\beta|>q$. Consequently, using Theorem \ref{Fourier},   careful computation reveals that
\begin{equation*}
\begin{split}
&\left<(I_\cE\otimes K_{f,rW}^{(m)})^*(T\otimes I_{F^2(H_n)})(I_\cE\otimes K_{f,rW}^{(m)})(x\otimes e_\gamma), y\otimes e_\omega\right>\\
&\qquad=\left<(T\otimes I_{F^2(H_n)})(x\otimes K_{f,rW}^{(m)}e_\gamma), y\otimes K_{f,rW}^{(m)}e_\omega\right>\\
&\qquad=
\left<(T\otimes I_{F^2(H_n)})\left( \sum_{\alpha\in \FF_n^+} x\otimes e_\alpha\otimes \Delta_{m,rW}^{1/2} W_\alpha^*(e_\gamma)\right), \sum_{\beta\in \FF_n^+}y\otimes e_\beta\otimes \Delta_{m,rW}^{1/2} W_\beta^* (e_\omega)\right>\\
&\qquad=
\left< \sum_{\alpha\in \FF_n^+} T( x\otimes e_\alpha)\otimes \Delta_{m,rW}^{1/2} W_\alpha^*(e_\gamma), \sum_{\beta\in \FF_n^+}y\otimes e_\beta\otimes \Delta_{m,rW}^{1/2} W_\beta^* (e_\omega)\right>\\
&\qquad=\sum_{\alpha\in \FF_n^+}\sum_{\beta\in \FF_n^+}\left<T(x\otimes e_\alpha), y\otimes e_\beta\right>\left<\Delta_{m,rW}^{1/2} W_\alpha^*(e_\gamma), \Delta_{m,rW}^{1/2} W_\beta^* (e_\omega)\right>\\
&\qquad=\sum_{\alpha\in \FF_n^+, |\alpha|\leq q}\sum_{\beta\in \FF_n^+, |\beta|\leq q}\left<T(x\otimes e_\alpha), y\otimes e_\beta\right>\left<\Delta_{m,rW}^{1/2} W_\alpha^*(e_\gamma), \Delta_{m,rW}^{1/2} W_\beta^* (e_\omega)\right>\\
&\qquad=\sum_{\alpha\in \FF_n^+, |\alpha|\leq q}\sum_{\beta\in \FF_n^+, |\beta|\leq q}\left<\varphi_q(W)(x\otimes e_\alpha), y\otimes e_\beta\right>\left<\Delta_{m,rW}^{1/2} W_\alpha^*(e_\gamma), \Delta_{m,rW}^{1/2} W_\beta^* (e_\omega)\right>\\
&\qquad=\sum_{\alpha\in \FF_n^+ }\sum_{\beta\in \FF_n^+ }\left<\varphi_q(W)(x\otimes e_\alpha), y\otimes e_\beta\right>\left<\Delta_{m,rW}^{1/2} W_\alpha^*(e_\gamma), \Delta_{m,rW}^{1/2} W_\beta^* (e_\omega)\right>\\
&\qquad=\left<(\varphi_q(W)\otimes I_{F^2(H_n)})(x\otimes K_{f,rW}^{(m)}e_\gamma), y\otimes K_{f,rW}^{(m)}e_\omega\right>\\
&\qquad =\left<(I_\cE\otimes K_{f,rW}^{(m)})^*(\varphi_q(W)\otimes I_{F^2(H_n)})(I_\cE\otimes K_{f,rW}^{(m)})(x\otimes e_\gamma), y\otimes e_\omega\right>\\
&\qquad =\left<\varphi_q(rW)(x\otimes e_\gamma), y\otimes e_\omega\right>\\
&\qquad =\left<\varphi(rW)(x\otimes e_\gamma), y\otimes e_\omega\right>
\end{split}
\end{equation*}
for any $x,y\in \cE$. This shows that relation \eqref{KTK} holds.
Since $K_{f,rW}^{(m)}$ is an isometry for any $r\in [0,1)$, we deduce that
\begin{equation}
\label{in-vN}
\|\varphi(rW)\|\leq \|T\|, \qquad r\in [0,1),
\end{equation}
which completes the proof of part (ii).
Now, we show that $T=\text{\rm SOT-}\lim\limits_{r\to 1} \varphi(rW)$.
Indeed, first note that, due to part (i) and the proof of Theorem \ref{Fourier}, we deduce that
$$
 \left\|\sum_{\alpha\in \FF_n^+} (A_{(\alpha)}\otimes W_\alpha)(x\otimes e_\beta)\right\|^2
 \leq \left(\begin{matrix} |\beta|+m-1\\ m-1\end{matrix}\right)\sum_{\alpha\in \FF_n^+} \|A_{(\alpha)}x\|^2 \frac{1}{b_{\alpha}^{(m)}}<\infty
 $$
 and also a similar result involving the other series in the definition of $\varphi(W)$.
Consequently,
\begin{equation}\label{fi-conv}
\|\varphi(rW)p-\varphi(W)p\|\to 0, \quad \text{ as } r\to 1,
\end{equation}
for any   $p:=\sum_{|\alpha|\leq k} h_{(\alpha)}\otimes e_\alpha$, where  $ h_{(\alpha)}\in \cE$ and $k\in \NN$. Let $x\in \cE\otimes F^2(H_n)$ and choose $p$ as above such that $\|x-p\|\leq \frac{\epsilon}{2\|T\|}$.
Using relation \eqref{in-vN} and Theorem \ref{Fourier}, we obtain
\begin{equation*}
\begin{split}
\|\varphi(rW)x-Tx\|&\leq \|\varphi(rW)(x-p)\|+\|\varphi(rW)p-\varphi(W)p\| +\|\varphi(W)p-Tx\|\\
&\leq 2\|T\|\|x-p\|+\|\varphi(rW)p-\varphi(W)p\|\\
&\leq \epsilon + \|\varphi(rW)p-\varphi(W)p\|.
\end{split}
\end{equation*}
Now, relation \eqref{fi-conv} implies
$\limsup_{r\to 1}\|\varphi(rW)x-Tx\|\leq \epsilon$ for any $\epsilon>0$. Hence $\lim_{r\to 1}\|\varphi(rW)x-Tx\|=0$ for any $x\in \cE\otimes F^2(H_n)$, which proves part (b).
\
To prove part (c), let $\epsilon>0$ and choose $p\in \cE\otimes F^2(H_n)$ be a polynomial such that $\|p\|=1$ and $\|Tp\|>\|T\|-\epsilon$.
Theorem \ref{Fourier} and relation \eqref{fi-conv} imply that there is $t\in (0,1)$ such that
$\|\varphi(tW)p\|>\|T\|-\epsilon$. This shows that
$\sup_{r\in [0,1)} \|\varphi(rW)\|\geq \|T\|$. Since the reverse inequality holds due to  relation \eqref{in-vN}, we deduce that
\begin{equation}\label{Tsup}
\sup_{r\in [0,1)} \|\varphi(rW)\|= \|T\|.
\end{equation}
Now, let $t_1,t_2\in [0,1)$ such that $t_1<t_2$. Since $ \varphi(t_2W) $ is in $\boldsymbol{\cA}_\cE({\bf D}_f^m)$ we can use the noncommutative Berezin transform to deduce that
$$
(I_\cE\otimes K_{f,rW}^{(m)})^*\left(\varphi(t_2W)\otimes I_{F^2(H_n)}\right)(I_\cE\otimes K_{f,rW}^{(m)})= \varphi(t_2rW)
$$
for any $r\in [0,1)$. Taking $r:=\frac{t_1}{t_2}$ and employing the fact that
$K_{f,rW}^{(m)}$ is an isometry, we obtain
$$
\|\varphi(t_1W)\|\leq \|\varphi(t_2W)\|,
$$
which together with  relation \eqref{Tsup} show that $\|T\|=\lim_{r\to 1}\|\varphi(rW)$. On the other hand, the fact that $\|T\|=\sup\limits_{q\in \cE\otimes\cP, \|q\|\leq 1}
\|\varphi(W)q\|$ is a consequence of Theorem \ref{Fourier}.

It remains to prove the converse of the theorem. Assume that
$\{A_{(\alpha)}\}_{\alpha\in \FF_n^+}$ and $\{B_{(\alpha)}\}_{\alpha\in
\FF_n^+\backslash \{g_0\}}$ are  two sequences of  operators on a Hilbert space $\cE$ satisfying conditions (i) and (ii), where $\varphi(W)$ is the formal series
$$
\sum_{|\alpha|\geq 1}B_{(\alpha)} \otimes W_\alpha^* +\sum_{\alpha\in \FF_n^+} A_{(\alpha)}\otimes W_\alpha.
$$
As is the first part of the proof, we can show that $\varphi(rW)$
 is in the operator space
$\boldsymbol{\cA}_\cE({\bf D}_f^m)^* +\boldsymbol{\cA}_\cE({\bf D}_f^m) $, and
$\varphi(W)p$ makes sense for any polynomial in $\cE\otimes F^2(H_n)$.
Note that item (ii) implies
\begin{equation} \label{sup-finite}
\sup\limits_{q\in \cE\otimes\cP, \|q\|\leq 1}
\|\varphi(W)q\|<\infty.
\end{equation}
Indeed, if $M>0$ and there is a polynomial  $p_0\in \cE\otimes \cP$ such that $\|\varphi(W)p_0\|>M$. Using the fact that $\|\varphi(rW)p_0-\varphi(W)p_0\|\to 0$ as $r\to 1$, we find $t\in (0,1)$ such that $\|\varphi(tW)p_0\|>M$, which implies that $\|\varphi(W)\|>M$. Since $M$ is arbitrary, we get a contradiction. Therefore, relation \eqref{sup-finite} holds.
Consequently, there is a unique operator $T\in B(\cE\otimes F^2(H_n))$ such that $Tp=\varphi(W)p$ for any polynomial $p\in \cE\otimes F^2(H_n)$. Now one can easily see that relation \eqref{fi}  holds and
\begin{equation*}
\begin{split}
\left<T(x\otimes e_\gamma), y\otimes e_\omega\right>&
=\left<\varphi(W)(x\otimes e_\alpha), y\otimes e_\omega\right>\\
&=\begin{cases}
\left<\frac{ \sqrt{b_\gamma^{(m)}}}
{\sqrt{b_{\alpha\gamma}^{(m)}}}A_{(\alpha)} x,y\right>,& if \ \omega=\alpha\gamma,
\\
\left<\frac{ \sqrt{b_\omega^{(m)}}}
{\sqrt{b_{\alpha\gamma}^{(m)}}}B_{(\alpha)} x,y\right>,& if \ \gamma=\alpha\omega,
\\
0,& otherwise.
\end{cases}\\
&=\begin{cases}\frac{\sqrt{b_{\omega\backslash_r\gamma}^{(m)}}\sqrt{b_\gamma^{(m)}}}
{\sqrt{b_\omega^{(m)}}}\left<T(x\otimes 1),y\otimes e_{\omega\backslash_r\gamma}\right>, & \text{if } \ \omega\geq_r\gamma,\\
\frac{\sqrt{b_{\gamma\backslash_r\omega}^{(m)}}\sqrt{b_\omega^{(m)}}}
{\sqrt{b_\gamma^{(m)}}}\left<T(x\otimes e_{\gamma\backslash_r\omega}),y\otimes 1\right>, & \text{if } \ \gamma >_r\omega,\\
0, & \text{otherwise}.
\end{cases}
\end{split}
\end{equation*}
This shows that $T$ is a weighted right multi-Toeplitz operator and completes the proof.
\end{proof}

\begin{corollary}\label{multi}
The set of all weighted right multi-Toeplitz operators  on $\cE\otimes F^2(H_n)$
coincides with
$$
 \overline{\boldsymbol{\cA}_\cE({\bf D}_f^m)^* +\boldsymbol{\cA}_\cE({\bf D}_f^m)}^{WOT}= \overline{\boldsymbol{\cA}_\cE({\bf D}_f^m)^* +\boldsymbol{\cA}_\cE({\bf D}_f^m)}^{SOT},
 $$
 where $\boldsymbol{\cA}_\cE({\bf D}_f^m):= B(\cE)\otimes_{min}\cA({\bf D}_f^m)$ and  $\cA({\bf D}_f^m)$ is the noncommutative domain
algebra.
\end{corollary}

\begin{proof}
Let $\cM_T$ be the set of all weighted right multi-Toeplitz  operators and note that the inclusion
$\cM_T\subseteq \overline{\boldsymbol{\cA}_\cE({\bf D}_f^m)^* +\boldsymbol{\cA}_\cE({\bf D}_f^m)}^{SOT}$ holds due to Theorem \ref{Toeplitz}. Since
$$
 \overline{\boldsymbol{\cA}_\cE({\bf D}_f^m)^* +\boldsymbol{\cA}_\cE({\bf D}_f^m)}^{SOT}\subseteq\overline{\boldsymbol{\cA}_\cE({\bf D}_f^m)^* +\boldsymbol{\cA}_\cE({\bf D}_f^m)}^{WOT},
 $$
 it remains to show that
 $$\overline{\boldsymbol{\cA}_\cE({\bf D}_f^m)^* +\boldsymbol{\cA}_\cE({\bf D}_f^m)}^{WOT}\subseteq \cM_T.
 $$
To this end, note that, for any operator $A\in B(\cE)$ and $\alpha\in \FF_n^+$, the operators  $A\otimes W_\alpha^*$ and  $A\otimes W_\alpha$ are multi-Toeplitz. On the other hand, if $\{T_i\}$ is a net of weighted right multi-Toeplitz operators such that $T_i\to T$ in the weak operator topology, passing to the limit in relation \eqref{MT}, written for $T_i$,  shows  that $T$ is a weighted right multi-Toeplitz operator  as  well. This completes the proof.
\end{proof}

Next, we show that for certain noncommutative domains ${\bf D}_f^m$, the corresponding set of weighted right multi-Toeplitz operators  does not contain any nonzero compact operator.
\begin{theorem} \label{compact} Let    ${\bf D}_f^m$ be a noncommutative domain where the coefficients $b_\alpha^{(m)}$ associated to $f$ satisfy the condition
$$
\sup_{\alpha\in \FF_n^+} \frac{b^{(m)}_{g_i\alpha }}{b^{(m)}_\alpha}<\infty,\qquad i\in \{1,\ldots, n\}.
$$
 Then there is no nonzero compact weighted right multi-Toeplitz operator  on the full Fock space $F^2(H_n)$.
\end{theorem}
\begin{proof}
First, note that
\begin{equation}
\label{sup}
\sup_{\alpha\in \FF_n^+} \frac{b^{(m)}_{\sigma\alpha }}{b^{(m)}_\alpha}<\infty\quad \text{ for each } \quad \sigma\in \FF_n^+.
\end{equation}
Indeed,
 if $\sigma=g_{i_1}\cdots g_{i_q}$, $i_1,\ldots, i_q\in \{1,\ldots,n\}$, then
$$
\frac{b_{\sigma \gamma}}{b_\gamma}=\frac{b_{g_{i_1}\cdots g_{i_q}\gamma}}{b_{g_{i_2}\cdots g_{i_q}\gamma}}
 =\frac{b_{g_{i_2}\cdots g_{i_q}\gamma}}{b_{g_{i_3}\cdots g_{i_q}\gamma}}\cdots \frac{b_{g_{i_q}}\gamma}{b_\gamma}.
$$
Using now the condition in the theorem, the assertion follows. Assume that $T$ is a compact  weighted right multi-Toeplitz operator on $F^2(H_n)$.
Then, we have
\begin{equation}
\label{Toep}
\begin{split}
\left<T e_\gamma, e_{\sigma \gamma}\right>&=\sqrt{\frac{b_\sigma^{(m)} b_\gamma^{(m)}}{b_{\sigma\gamma}^{(m)}}}\left<T e_{g_0}, e_\sigma\right>\\
\left<Te_{\sigma \gamma}, e_\gamma\right>&=\sqrt{\frac{b_\sigma^{(m)} b_\gamma^{(m)}}{b_{\sigma\gamma}^{(m)}}}\left<T  e_\sigma, e_{g_0}\right>
\end{split}
\end{equation}
for any $\sigma\gamma\in \FF_n^+$, and $\left<T e_\alpha, e_\beta\right>=0$  if $(\alpha, \beta) \in \FF_n^+\times \FF_n^+$ is not of the form  $(\sigma\gamma, \gamma)$ or $(\gamma, \sigma\gamma)$ for   $\sigma, \gamma\in \FF_n^+$.
Since a compact operator maps weakly convergent sequences to norm convergent sequences, we  deduce that
$\left<T e_\gamma, e_{\sigma \gamma}\right>\to 0$ and  $\left<Te_{\sigma \gamma}, e_\gamma\right> \to 0$ as $|\gamma|\to \infty$. Consequently, using relations \eqref{sup} and \eqref{Toep}, we conclude that
$\left<T e_\gamma, e_{\sigma \gamma}\right>=\left<Te_{\sigma \gamma}, e_\gamma\right>=0$ for any $\sigma\in \FF_n^+$. Now, using again relation \eqref{Toep}, we deduce that $T=0$, which completes the proof.
\end{proof}

    We  shall present a concrete class of noncommutative domains for which  the theorem above  holds. Consider  the  case when $f=Z_1+\cdots +Z_n$ and $m\in \NN$.   The corresponding domain ${\bf D}_f^m$ is  the {\it noncommutative $m$-hyperball}, which is defined by
$$
\left\{ X:=(X_1,\ldots, X_n)\in B(\cH)^n: \ (id-\Phi_{X})^k(I)\geq 0 \  \text{\rm  for } 1\leq k\leq m\right\},
$$
where $\Phi_{X}: B(\cH)\to B(\cH)$ is defined by $\Phi_{X}(Y):=\sum_{i=1}^n X_i Y X_i^*$ for  $Y\in B(\cH)$.
 In this case, we have    $b_{g_0}^{(m)} =1$ and
$
 b_\alpha^{(m)} =
\left(\begin{matrix} |\alpha|+m-1\\m-1
\end{matrix}\right)$
if   \ $ \alpha\in \FF_n^+,  |\alpha|\geq 1.
$
 Consequently,
 $$
 \frac{b^{(m)}_{\alpha g_i}}{b^{(m)}_\alpha}
 =\frac{\left(\begin{matrix} |\alpha|+m\\m-1
\end{matrix}\right)}{\left(\begin{matrix} |\alpha|+m-1\\m-1
\end{matrix}\right)}\to 1,\quad \text{ as } \ |\alpha|\to \infty.
 $$
This shows that Theorem \ref{compact} holds for the  weighted right multi-Toeplitz operators  associated with the noncommutative $m$-hyperball.

\bigskip

\section{Free pluriharmonic functions on the noncommutative domain ${\bf D}_{f,rad}^m$}

In this section, we   provide basic results concerning the free pluriharmonic functions on the  noncommutative domain ${\bf D}_{f, rad}^m(\cH)$ and show that they are characterized by a mean value property. This result is used to obtain an analogue of Weierstrass theorem for free pluriharmonic functions and to show that the set   of all pluriharmonic functions is a complete metric space with respect to an appropriate metric.  A Schur type result in this setting  is also presented.

Since the domain  ${\bf D}_f^m$ is radial (see \cite{Po-Berezin1}),  i.e. $r{X}\in {\bf D}_f^m(\cH)$ for any ${X}\in {\bf D}_f^m(\cH)$ and any $r\in [0,1)$,
 we can introduce   the radial part of the domain ${\bf D}_f^m(\cH)$, i.e.
$$
{\bf D}_{f,\text{\rm rad}}^m(\cH):= \bigcup_{0\leq t<1} t{\bf D}_{f}^m(\cH)\subseteq {\bf D}_{f}^m(\cH).
$$
Note that, in general, we have
$$
\text{\rm Int}{\bf D}_f^m(\cH)\subseteq {\bf D}_{f,\text{\rm rad}}^m(\cH)\subseteq {\bf D}_f^m(\cH)\subseteq {\bf D}_{f,\text{\rm rad}}^m(\cH)^-.
$$
In the particular case when $q$ is a positive regular noncommutative polynomial, we have
$$
\text{\rm Int}{\bf D}_q^m(\cH)={\bf D}_{q,\text{\rm rad}}^m(\cH) \quad \text{\rm and }\quad  {\bf D}_{q,\text{\rm rad}}^m(\cH)^-={\bf D}_q^m(\cH)={\bf D}_q^m(\cH)^-.
$$

Let $Z_1,\ldots, Z_n$ be noncommutative  indeterminates, set $Z_\alpha:=Z_{i_1}\cdots Z_{i_k}$ if $\alpha=g_{i_1}\cdots g_{i_k}\in \FF_n^+$, and $Z_{g_0}:=1$.

 \begin{definition} \label{def-free-hol} A formal power series  $F:=\sum\limits_{\alpha\in \FF_{n}} A_{(\alpha)}\otimes {Z}_{\alpha}$   with $A_{(\alpha)}\in B(\cE)$
 is called  free holomorphic function on the
  abstract   domain
${\bf D}_{f,\text{\rm rad}}^m:=
\coprod_{\cH}{\bf D}_{f,\text{\rm rad}}^m(\cH)$,  if the series

$$
F({X} ):=\sum\limits_{k=1}^\infty\sum\limits_{\alpha\in  \FF_n^+, |\alpha|=k} A_{(\alpha)}\otimes  {X}_{\alpha}
$$
is convergent in the operator norm topology for any $ {X} \in {\bf D}_{f,\text{\rm rad}}^m(\cH)$    and any Hilbert space $\cH$.
\end{definition}
We remark that it is enough to assume in Definition \ref{def-free-hol} that $\cH$ is an arbitrary infinite dimensional separable Hilbert space. Unless otherwise specified, we assume throughout this paper that $\cH$ has this property.

We denote by $Hol_\cE({\bf D}_f^m)$ the set of all free holomorphic functions on the
  abstract   domain
${\bf D}_{f,\text{\rm rad}}^m$ with operator coefficients in $B(\cE)$.
Let $F:=\sum\limits_{\alpha\in \FF_{n}} A_{(\alpha)}\otimes {Z}_{\alpha}$   with $A_{(\alpha)}\in B(\cE)$  be a formal power series
 and define
 $\gamma\in[0,\infty]$ by setting
 $$
 \frac{1}{\gamma}:=\limsup_{k\in \ZZ_+} \left\|\sum\limits_{\alpha\in  \FF_n^+, |\alpha|=k}A_{(\alpha)}\otimes { W}_{\alpha}\right\|^{\frac{1}{k}}.
 $$
 According to \cite{Po-Bohr2}, the series
 $$
 \sum\limits_{k=1}^\infty\left\|\sum\limits_{|\alpha|=k}A_{(\alpha)}\otimes { X}_{\alpha}\right\|, \qquad {X}\in   \gamma{\bf D}_{f,\text{\rm rad}}^m(\cH),
 $$
 is convergent. Moreover, if $\gamma>0$ and $r\in [0,\gamma)$, then the convergence is uniform on $r{\bf D}_f^m(\cH)$.
 In addition, if $\gamma\in [0,\infty)$ and
   $s>\gamma$, then there is a Hilbert space $\cH$ and ${Y}\in s{\bf D}_f^m(\cH)$  such that the series
 $$
 \sum\limits_{k=1}^\infty\sum\limits_{|\alpha|=k}A_{(\alpha)}\otimes { Y}_{\alpha}
 $$
 is divergent in the operator norm topology. As a consequence of these results, we deduce the following.
 \begin{proposition} \label{cha-hol}
 $F=\sum\limits_{\alpha\in \FF_{n}} A_{(\alpha)}\otimes {Z}_{\alpha}$  is free holomorphic on ${\bf D}_{f,\text{\rm rad}}^m$ if and only if
  $$
  \limsup_{k\in \ZZ_+} \left\|\sum\limits_{\beta\in  \FF_n^+, |\beta|=k}\omega_\beta A_{(\beta)}^*A_{(\beta)}\right\|^{1/2k}\leq 1,
  $$
  where $\omega_\beta:=\sup_{\gamma\in \FF_n^+} \frac{b_\gamma^{(m)}}{b_{\beta \gamma}^{(m)}}$ and $b_\gamma^{(m)}$ is given by relation \eqref{b-al}.
 \end{proposition}
 \begin{proof}
 Due to the results preceding the proposition, we have that $F$ is free holomorphic on ${\bf D}_{f,\text{\rm rad}}^m$ if and only if
 $\limsup_{k\in \ZZ_+} \left\|\sum\limits_{\alpha\in  \FF_n^+, |\alpha|=k}A_{(\alpha)}\otimes { W}_{\alpha}\right\|^{\frac{1}{k}}\leq 1$.
 On the other hand,  due to relation \eqref{WbWb},  we have
 $$
 W_\beta W_\beta^* e_\alpha =\begin{cases} \frac
{\sqrt{b_\gamma^{(m)}}}{\sqrt{b_{\alpha}^{(m)}}}e_\alpha& \text{ if
}
\alpha= \beta \gamma, \\
0& \text{ otherwise},
\end{cases}
$$
which implies $\|W_\beta W_\beta^*\|=\omega_\beta$. Since  the operators $W_\beta$, with $\beta\in \FF_n^+$  and $ |\beta|=k$, have orthogonal  ranges, we deduce that
$$
\left\|\sum\limits_{\beta\in  \FF_n^+, |\beta|=k}A_{(\beta)}\otimes { W}_{\beta}\right\|=\left\|\sum_{\beta\in \FF_n^+, |\beta|=k} A_{(\beta)}^* A_{(\beta)}\otimes W_\beta W_\beta^*\right\|^{1/2}=  \left\|\sum\limits_{\beta\in  \FF_n^+, |\beta|=k}\omega_\beta A_{(\beta)}^*A_{(\beta)}\right\|^{1/2}.
$$
The proof is complete.
 \end{proof}

\begin{definition}  \label{def-pluri} A map $G:{\bf D}_{f, rad}^m(\cH)\to
 B(\cE)\otimes_{min} B( \cH)$ is called
self-adjoint free pluriharmonic function  on ${\bf D}_{f, rad}^m(\cH)$ with
coefficients in $B(\cE)$ if there is a free holomorphic function $F$ on ${\bf D}_{f, rad}^m(\cH)$ such that $G=\Re F$. Any linear combination of self-adjoint free pluriharmonic functions  is called free pluriharmonic function.
\end{definition}
We remark that if $G=\Re F$ as in the latter definition, then $G$ determines $F$ up to an  operator $A_{(g_0)}\in B(\cE)$ with $\Re A_{(g_0)}=0$. Indeed, assume that $\Re F=0$. Then $F(rW)=-F(rW)^*$, $r\in [0,1)$. If  $F$ has the representation $F=\sum_{\alpha\in \FF_n^+} A_{(\alpha)}\otimes  Z_\alpha$, the relation above implies
$$F(rW)(x\otimes 1)=-F(rW)^*(x\otimes 1)=-A_{(g_0)}^*x,\qquad x\in \cE.
 $$
 Hence, we deduce that $A_{(\alpha)}=0$ if $\alpha\in \FF_n^+$ with $|\alpha|\geq 1$ and $\Re A_{(g_0)}=0$. Therefore, $F=A_{(g_0)}\otimes I$.
 On the other hand, it is easy to see that any free pluriharmonic function $H$ has a representation of the form $H=H_1+iH_2$, where $H_1$ and $H_2$ are self-adjoint free pluriharmonic  functions.
Note also that any free holomorphic function $F$ is a free pluriharmonic function, due to the decomposition $F=\frac{F+F^*}{2}+i\frac{F-F^*}{2i}$.

Using Proposition \ref{cha-hol}, one can easily prove the following characterization of free pluriharmonic functions on ${\bf D}_{f, rad}^m$.

\begin{proposition} \label{ser-def}
 A map $G:{\bf D}_{f, rad}^m(\cH)\to
 B(\cE)\otimes_{min} B( \cH)$ is a
free pluriharmonic function  on ${\bf D}_{f, rad}^m(\cH)$ with
coefficients in $B(\cE)$ if and only if there exist two sequences
 $\{A_{(\alpha)}\}_{\alpha\in \FF_n^+}\subset B(\cE)$ and
 $\{B_{(\alpha)}\}_{\alpha\in \FF_n^+\backslash \{g_0\}}\subset B(\cE)$
such that
$$
\limsup_{k\in \ZZ_+} \left\|\sum\limits_{\beta\in  \FF_n^+, |\beta|=k}\omega_\beta A_{(\beta)}^*A_{(\beta)}\right\|^{1/2k}\leq 1\quad \text{ and }\quad \limsup_{k\in \ZZ_+} \left\|\sum\limits_{\beta\in  \FF_n^+, |\beta|=k}\omega_\beta B_{(\beta)}B_{(\beta)}^*\right\|^{1/2k}\leq 1
$$
and
\begin{equation*} G(X)=\sum_{k=1}^\infty \sum_{|\alpha|=k} B_{(\alpha)}\otimes
X_\alpha^* +A_{(0)}\otimes  I+ \sum_{k=1}^\infty
 \sum_{|\alpha|=k} A_{(\alpha)}\otimes  X_\alpha,
\end{equation*}
where the series are convergent in the  operator norm topology
  for any $X\in {\bf D}_{f, rad}^m(\cH)$ and any Hilbert space $\cH$. Moreover, the representation of $G$ is unique.
\end{proposition}

The   {\it extended noncommutative Berezin transform at} $X\in {\bf D}_f^m(\cH)$, where $ X$ is a pure element,
 is the map $\widetilde{\bf B}^{(m)}_X: B(\cE\otimes F^2(H_{n}))\to B(\cE\otimes \cH)$
 defined by
 \begin{equation*}
 \widetilde{\bf B}^{(m)}_X[g]:= \left(I_\cE\otimes {K_{f,X}^{(m)}}^*\right) (g\otimes I_\cH)\left(I_\cE\otimes K_{f,X}^{(m)}\right),
 \quad g\in B(\cE\otimes F^2(H_n)),
 \end{equation*}
where the  $K_{f,X}^{(m)}:\cH \to F^2(H_n)\otimes
\cH$  is noncommutative Berezin kernel.

We denote by $\cP_\cE(W)$ the set of all operators of the form
$  \sum_{|\alpha|\leq k}A_{(\alpha)}\otimes W_\alpha$, where $k\in \NN$ and  $A_{(\alpha)}\in B(\cE)$. The following result extends Theorem 2.4 from \cite{Po-Berezin1}. We include it for completeness.

\begin{theorem} \label{funct-calc}
 If $X\in {\bf D}_f^m(\cH)$ and  $\cS_\cE:=\overline{\cP_\cE(W)^*+\cP_\cE(W)}^{\|\cdot\|}$,
then there is a unital completely contractive linear map $\Phi_{f,X}:\cS_\cE\to B(\cE)\otimes_{min} B(\cH)$ such that
$$
\Phi_{f,X}(\varphi)=\lim_{r\to 1}\widetilde {\bf B}_{rX}^{(m)}[\varphi], \qquad \varphi\in \cS_\cE,
$$
where the limit is in the operator norm topology. If, in addition, $X$ is a pure $n$-tuple in ${\bf D}_f^m(\cH)$, then
$$
\lim_{r\to 1}\widetilde {\bf B}_{rX}^{(m)}[\varphi]=\widetilde{\bf B}_X^{(m)}[\varphi],\qquad  \varphi\in \cS_\cE.
$$
\end{theorem}
\begin{proof} Let $\varphi\in \cS_\cE$ and let $\{q_k(W, W^*)\}_{k=1}^\infty\subset \cP_\cE(W)^*+\cP_\cE(W)$ be such that
$q_k(W, W^*)\to \varphi$ in the operator norm, as $k\to \infty$.
For any $X\in {\bf D}_f^m(\cH)$, the noncommutative von Neumann inequality (see \cite{Po-domains-models})  implies $\|q_k(X, X^*)-q_j(X,X^*)\|\leq \|q_k(W, W^*)-q_j(W,W^*)\|$ for any $k,j\in \NN$.
Consequently, since $\{q_k(W, W^*)\}_{k=1}^\infty$ is a Cauchy sequence, so is the sequence
$\{q_k(X, X^*)\}_{k=1}^\infty$. Therefore,
\begin{equation}
\label{def-Phi}
\Phi_{f,X}(\varphi):=\lim_{k\to \infty}q_k(X, X^*), \qquad X\in {\bf D}_f^m(\cH),
\end{equation}
exists in the operator norm and $\|\Phi_{f,X}(\varphi)\|\leq \|\varphi\|$.
Now, we show that
$$
\Phi_{f,X}(\varphi)=\lim_{r\to 1}\widetilde {\bf B}_{rX}^{(m)}[\varphi], \qquad \varphi\in \cS_\cE,
$$
where the limit is in the operator norm topology. Since
$$
\widetilde {\bf B}_{rX}^{(m)}[q_k(W,W^*)]=
(I_\cE\otimes K_{f,rX}^{(m)})^*(q_k(W,W^*)\otimes I_\cH)(I_\cE\otimes  K_{f,rX}^{(m)})=q_k(rX,rX^*),
$$
relation \eqref{def-Phi} implies
\begin{equation}
\label{PhiB}
\Phi_{f,rX}(\varphi)=\widetilde {\bf B}_{rX}^{(m)}[\varphi],\qquad r\in [0,1).
\end{equation}
Given $\epsilon>0$, let $N\in \NN$ be such that $\|q_N(W,W^*)-\varphi\|<\frac{\epsilon}{3}$. Note that
$$
\Phi_{f,rX}(\varphi)-q_N(rX,rX^*)=\lim_{k\to \infty} q_k(rX,rX^*)-q_N(rX,rX^*)
$$
and
\begin{equation}
\label{one}
\|\Phi_{f,rX}(\varphi)-q_N(rX,rX^*)\|\leq \|\varphi-q_N(W,W^*)\|<\frac{\epsilon}{3},\qquad r\in [0,1].
\end{equation}
Note also that there is a $\delta\in (0,1)$ such that
\begin{equation}
\label{two}
\|q_N(rX,rX^*)-q_N(X,X^*)\|<\frac{\epsilon}{3},\qquad r\in (\delta,1).
\end{equation}
Using the relations \eqref{def-Phi}, \eqref{PhiB},  \eqref{one} and \eqref{two}, we obtain
\begin{equation*}
\begin{split}
\|\Phi_{f,X}(\varphi)-\widetilde {\bf B}_{rX}^{(m)}[\varphi]\|
&=
\|\Phi_{f,X}(\varphi)- \Phi_{f,rX}(\varphi)\|\\
&\leq \|\Phi_{f,X}(\varphi)-q_N(X,X^*)\|+ \|q_N(X,X^*)-q_N(rX,rX^*)\|\\
&\qquad \qquad+\|q_N(rX,rX^*)-\Phi_{f,rX}(\varphi)\|<\epsilon
\end{split}
\end{equation*}
for any $r\in (\delta, 1)$. Therefore, $$
\Phi_{f,X}(\varphi)=\lim_{r\to 1}\widetilde {\bf B}_{rX}^{(m)}[\varphi], \qquad \varphi\in \cS_\cE,
$$
where the limit is in the operator norm topology. Similarly, one can prove that, for any $k\times k$ matrix $[\varphi_{ij}]_{k\times k}$ with entries in $\cS_\cE$,
$$
\left[\Phi_{f,X}(\varphi_{ij})\right]_{k\times k}=\lim_{r\to 1}\left[\widetilde {\bf B}_{rX}^{(m)}[\varphi_{ij}]\right]_{k\times k}
$$
in the operator norm. Using the properties of the noncommutative Berezin kernel, we deduce that
$$
\left\|\left[\Phi_{f,X}(\varphi_{ij})\right]_{k\times k}\right\|\leq \left\|[\varphi_{ij}]_{k\times k}\right\|.
$$
This proves that $\Phi_{f,X}$ is a unital completely contractive linear map.
Now, assume that $X\in {\bf D}_f^m(\cH)$ is a pure $n$-tuple of operators. Then
we have
$$
\widetilde {\bf B}_{X}^{(m)}[q_k(W,W^*)]=(I_\cE\otimes K_{f,X}^{(m)})^*(q_k(W,W^*)\otimes I_\cH)(I_\cE\otimes  K_{f,X}^{(m)})=q_k(X,X^*).
$$
Since $q_k(W, W^*)\to \varphi$ in the operator norm, as $k\to \infty$, we can use the relation above and \eqref{def-Phi} to deduce that
$$\widetilde {\bf B}_{X}^{(m)}[\varphi]=\lim_{k\to\infty} q_k(X,X^*)=\Phi_{f,X}(\varphi)=\lim_{r\to 1}
\widetilde {\bf B}_{rX}^{(m)}[\varphi].
$$
The proof is complete.
\end{proof}
Next, we show that the free pluriharmonic functions are characterized by a mean value property.

\begin{theorem} \label{mean-value} If $G:{\bf D}_{f, rad}^m(\cH)\to
 B(\cE)\otimes_{min} B( \cH)$ is a
free pluriharmonic function, then it has the mean value property, i.e
$$
G(X)=\widetilde{\bf B}^{(m)}_{\frac{1}{r}X}[G(rW)],\qquad X\in r{\bf D}_f^m(\cH), r\in (0,1).
$$

Conversely, if a function $\varphi:[0,1)\to \overline{\boldsymbol{\cA}_n({\bf D}_f^m)^*+\boldsymbol{\cA}_n({\bf D}_f^m)}^{\|\cdot\|}$ satisfies the relation
$$
\varphi(r)=\widetilde{\bf B}^{(m)}_{\frac{r}{t}W}[\varphi(t)],\qquad \text{for any }\  0\leq r<t<1,
$$
then the map $F:{\bf D}_{f, rad}^m(\cH)\to
 B(\cE)\otimes_{min} B( \cH)$ defined by
 $$
 F(X):=\widetilde{\bf B}^{(m)}_{\frac{1}{r}X}[\varphi(r)],\qquad X\in r{\bf D}_f^m(\cH), r\in (0,1),
$$
  is a
free pluriharmonic function.
Moreover, $F(rW)=\varphi(r)$ for any $r\in [0,1)$. In particular, $F\geq 0$ if and only if $\varphi\geq 0$.
 \end{theorem}
\begin{proof}
Assume that $G:{\bf D}_{f, rad}^m(\cH)\to
 B(\cE)\otimes_{min} B( \cH)$ is a
free pluriharmonic function with representation
\begin{equation*} G(X)=\sum_{k=1}^\infty \sum_{|\alpha|=k} B_{(\alpha)}\otimes
X_\alpha^* +A_{(0)}\otimes  I+ \sum_{k=1}^\infty
 \sum_{|\alpha|=k} A_{(\alpha)}\otimes  X_\alpha,
\end{equation*}
where the series are convergent in the  operator norm topology
  for any $X\in {\bf D}_{f, rad}^m(\cH)$ and any Hilbert space $\cH$.
Since the universal model $W=(W_1,\ldots, W_n)$ is in ${\bf D}_f^m(F^2(H_n))$, for any $r\in [0,1)$, we have

\begin{equation*} G(rW)=\sum_{k=1}^\infty \sum_{|\alpha|=k} B_{(\alpha)}\otimes r^{|\alpha|}
W_\alpha^* +A_{(0)}\otimes  I+ \sum_{k=1}^\infty
 \sum_{|\alpha|=k} A_{(\alpha)}\otimes   r^{|\alpha|}W_\alpha,
\end{equation*}
where the convergence is in the operator norm topology.
If $X\in {\bf D}_{f, rad}^m(\cH)$, $r\in (0,1)$, then $\frac{1}{r} X\in {\bf D}_f^m(\cH)$ and
\begin{equation*}
\begin{split}
\widetilde{\bf B}^{(m)}_{\frac{1}{r}X}[G(rW)]&=\lim_{\delta\to 1}\widetilde{\bf B}^{(m)}_{\frac{\delta}{r}X}[G(rW)]\\
&=\lim_{\delta\to 1} G(\delta X)=G(X),
\end{split}
\end{equation*}
where the limits are in the operator norm topology. The latter equality is due to the continuity of $G$ on  $r{\bf D}_{f}^m(\cH)$.
To prove the converse, assume that  the function $\varphi:[0,1)\to \overline{\boldsymbol{\cA}_\cE({\bf D}_f^m)^*+\boldsymbol{\cA}_\cE({\bf D}_f^m)}^{\|\cdot\|}$ satisfies the relation
$$
\varphi(r)=\widetilde{\bf B}^{(m)}_{\frac{r}{t}W}[\varphi(t)],\qquad \text{for any }\  0\leq r<t<1.
$$
Due to Theorem  \ref{Toeplitz} and Corollary \ref{multi}, $\varphi(r)$ is a  weighed right multi-Toeplitz operator and has a unique formal  Fourier  representation
$$\sum_{|\alpha|\geq 1} r^{|\alpha|} B_{(\alpha)}(r)\otimes W_\alpha^* +
\sum_{|\alpha|\geq 0} r^{|\alpha|} A_{(\alpha)}(r)\otimes W_\alpha
$$
for some operators $\{A_{(\alpha)}(r)\}_{\alpha\in \FF_n^+}$ and $\{B_{(\alpha)}(r)\}_{\alpha\in
\FF_n^+\backslash \{g_0\}}$.
Moreover, setting
$$
\varphi_\delta(r):=\sum_{|\alpha|\geq 1} r^{|\alpha|} B_{(\alpha)}(r)\otimes \delta^{|\alpha|}W_\alpha^* +
\sum_{|\alpha|\geq 0} r^{|\alpha|} A_{(\alpha)}(r)\otimes \delta^{|\alpha|} W_\alpha, \qquad \delta\in [0,1),
$$
where  the convergence of the series is in the operator norm topology,
we have
\begin{equation}
\label{Vasu}
\varphi(r)=\text{\rm SOT-}\lim_{\delta\to 1}\varphi_\delta(r)\quad \text{ and }\quad  \sup_{\delta\in [0,1)}\|\varphi_\delta(r)\|=\|\varphi(r)\|.
\end{equation}
Due to similar reasons,
$\varphi(t)$ is a weighted right multi-Toeplitz operator and has a unique formal  Fourier  representation
$$\sum_{|\alpha|\geq 1} t^{|\alpha|} B_{(\alpha)}(t)\otimes W_\alpha^* +
\sum_{|\alpha|\geq 0} t^{|\alpha|} A_{(\alpha)}(t)\otimes W_\alpha
$$
for some operators $\{A_{(\alpha)}(t)\}_{\alpha\in \FF_n^+}$ and $\{B_{(\alpha)}(t)\}_{\alpha\in
\FF_n^+\backslash \{g_0\}}$.
Moreover, setting
$$
\psi_\delta(t):=\sum_{|\alpha|\geq 1} t^{|\alpha|} B_{(\alpha)}(t)\otimes \delta^{|\alpha|}W_\alpha^* +
\sum_{|\alpha|\geq 0} t^{|\alpha|} A_{(\alpha)}(t)\otimes \delta^{|\alpha|} W_\alpha, \qquad \delta\in [0,1),
$$
where  the convergence of the series is in the operator norm topology,
we have
$$\varphi(t)=\text{\rm SOT-}\lim_{\delta\to 1}\psi_\delta(t)\quad \text{ and }\quad  \sup_{\delta\in [0,1)}\|\psi_\delta(t)\|=\|\varphi(t)\|.
$$
Now, since the map $Y\to Y\otimes I$ is SOT-continuous on bounded sets, so is the noncommutative Berezin transform $\widetilde{\bf B}^{(m)}_{\frac{r}{t} W}$.
Using the results above,  we deduce that
\begin{equation*}
\begin{split}
\widetilde{\bf B}^{(m)}_{\frac{r}{t}W}[\varphi(t)]
&=\text{\rm SOT-}\lim_{\delta \to 1}
\widetilde{\bf B}^{(m)}_{\frac{r}{t}W}[\psi_\delta(t)]\\
&=\text{\rm SOT-}\lim_{\delta \to 1}\left(
\sum_{|\alpha|\geq 1} r^{|\alpha|} B_{(\alpha)}(t)\otimes \delta^{|\alpha|}W_\alpha^* +
\sum_{|\alpha|\geq 0} r^{|\alpha|} A_{(\alpha)}(t)\otimes \delta^{|\alpha|} W_\alpha
\right).
\end{split}
\end{equation*}
Consequently, using the fact that  $\varphi(r)=\widetilde{\bf B}^{(m)}_{\frac{r}{t}W}[\varphi(t)]$ and relation \eqref{Vasu}, we can easily  see that $B_{(\alpha)}(t)=B_{(\alpha)}(r)$ for any $\alpha\in \FF_n^+$ with $|\alpha|\geq 1$ and  $A_{(\alpha)}(t)=A_{(\alpha)}(r)$ for any $\alpha\in \FF_n^+$. Therefore,
$$
\varphi(r)=\sum_{|\alpha|\geq 1} r^{|\alpha|} B_{(\alpha)}\otimes W_\alpha^* +
\sum_{|\alpha|\geq 0} r^{|\alpha|} A_{(\alpha)}\otimes W_\alpha, \qquad r\in [0,1),
$$
for some operators $\{A_{(\alpha)}\}_{\alpha\in \FF_n^+}$ and $\{B_{(\alpha)}\}_{\alpha\in
\FF_n^+\backslash \{g_0\}}$, where the convergence is in the operator norm topology. Now, for any $X\in r{\bf D}_f^m(\cH)$, $r\in (0,1)$,  we define
$$
 F(X):=\widetilde{\bf B}^{(m)}_{\frac{1}{r}X}[\varphi(r)].
$$
Hence, we deduce that
$$
F(X)= \sum_{|\alpha|\geq 1}  B_{(\alpha)}\otimes X_\alpha^* +
\sum_{|\alpha|\geq 0}  A_{(\alpha)}\otimes X_\alpha, \qquad X\in {\bf D}_{f,rad}^m(\cH),
$$
and $F(rW)=\varphi(r)$ for any $r\in [0,1)$. The proof is complete.
\end{proof}

Let $Hol_\cE^+({\bf D}_{f,rad}^m)$ be the set of all free holomorphic functions $F\in Hol_\cE({\bf D}_{f,rad}^m)$ such that $\Re F\geq 0$.
 If $F:=\sum\limits_{\alpha\in \FF_{n}} A_{(\alpha)}\otimes {Z}_{\alpha}$, we associated the kernel $\Gamma_{rF}:\FF_n^+\times \FF_n^+\to B(\cE)$ defined by

 \begin{equation*}
\Gamma_{rF}(\omega, \gamma):=
\begin{cases}
 \frac{ \sqrt{b_\gamma^{(m)}}}
{\sqrt{b_{\alpha\gamma}^{(m)}}}r^{|\alpha|}A_{(\alpha)},& if \ \omega=\alpha\gamma,\, |\alpha|\geq 1,
\\
A_{(g_0)}+A_{(g_0)}^*,& if \omega=\gamma,
\\
 \frac{ \sqrt{b_\omega^{(m)}}}
{\sqrt{b_{\alpha\omega}^{(m)}}}r^{|\alpha|}A_{(\alpha)}^*,& if \ \gamma=\alpha\omega,\,  |\alpha|\geq 1,
\\
0,& otherwise.
\end{cases}
\end{equation*}

 We will use the notation $\cS^+_\cE({\bf D}_{f,rad}^m)$ for the set of all free holomorphic functions $F\in Hol_\cE({\bf D}_{f,rad}^m)$ such that the weighted multi-Toeplitz kernels
$\Gamma_{rF}:\FF_n^+\times \FF_n^+\to B(\cE)$, $r\in [0,1)$,  are positive semidefinite.

Now, we prove a Schur type result for free pluriharmonic functions with positive real parts.

\begin{theorem}
$Hol_\cE^+({\bf D}_{f,rad}^m)=\cS^+_\cE({\bf D}_{f,rad}^m)$.
\end{theorem}
\begin{proof}
Assume that $F\in Hol_\cE^+({\bf D}_{f,rad}^m)$ has the representation
$F=\sum_{\alpha\in \FF_n^+} A_{(\alpha)}\otimes Z_\alpha$ and let $h_\beta\in \cE$, for $\beta\in \FF_n^+$ with $|\beta|\leq q$. Note that, using relation \eqref{WbWb}, we have
\begin{equation*}
\begin{split}
&\left<\sum_{s=1}^\infty\sum_{|\alpha|=s} A_{(\alpha)}\otimes r^{|\alpha|} W_\alpha \left(\sum_{|\beta|\leq q} h_\beta\otimes e_\beta\right), \sum_{|\gamma|\leq q} h_\gamma\otimes e_\gamma\right>\\
&\qquad =\sum_{s=1}^\infty\sum_{|\alpha|=s}\left<\sum_{|\beta|\leq q}A_{(\alpha)}h_\beta\otimes r^{|\alpha|} W_\alpha e_\beta, \sum_{|\gamma|\leq q} h_\gamma \otimes e_\gamma\right>\\
&\qquad =
\sum_{\alpha\in \FF_n^+, |\alpha|\geq 1} \sum_{|\beta|, |\gamma|\leq q}r^{|\alpha|}\left<A_{(\alpha)} h_\beta, h_\gamma\right>\left< \frac{\sqrt{b_\beta^{(m)}}}{\sqrt{b_{\alpha \beta}^{(m)}}}e_{\alpha\beta}, e_\gamma\right>\\
&\qquad =
 \sum_{|\beta|, |\gamma|\leq q, \gamma >_r \beta}r^{|\gamma \backslash_r \beta|}\left<A_{(\gamma \backslash_r \beta)} h_\beta, h_\gamma\right>  \frac{\sqrt{b_\beta^{(m)}}}{\sqrt{b_{\gamma}^{(m)}}} \\
 &\qquad =
  \sum_{|\beta|, |\gamma|\leq q, \gamma >_r \beta}
  \left<\Gamma_{rF}(\gamma, \beta) h_\beta, h_\gamma\right>.
\end{split}
\end{equation*}
In a similar manner, one can prove that
$$
\left<\sum_{s=1}^\infty\sum_{|\alpha|=s} A_{(\alpha)}\otimes r^{|\alpha|} W_\alpha^* \left(\sum_{|\beta|\leq q} h_\beta\otimes e_\beta\right), \sum_{|\gamma|\leq q} h_\gamma\otimes e_\gamma\right>
=
\sum_{|\beta|, |\gamma|\leq q, \gamma >_r \beta}
  \left<\Gamma_{rF}(\gamma, \beta) h_\beta, h_\gamma\right>.
$$
Note also that, for any $\beta\in \FF_n$, $\Gamma_{rF}(\beta, \beta)=\Gamma_{rF}(g_0, g_0)=A_{g_0}+A_{g_0}^*$.Taking into account  the relations above, we deduce that
$$
\left<(F(rW)^*+F(rW))\left(\sum_{|\beta|\leq q} h_\beta\otimes e_\beta\right), \sum_{|\gamma|\leq q} h_\gamma\otimes e_\gamma\right>
=
\sum_{|\beta|, |\gamma|\leq q}
  \left<\Gamma_{rF}(\gamma, \beta) h_\beta, h_\gamma\right>.
$$
Consequently, $F(rW)^*+F(rW)\geq 0$ for any $r\in [0,1)$ if and only if $\Gamma_{rF}$ is a positive semidefinite kernel for any $r\in [0,1)$.

 On the other hand, if $X\in {\bf D}^m_{f,rad}(\cH)$, then there is $r\in (0,1)$ such that $X\in r{\bf D}^m_{f}(\cH)$. Due to Theorem \ref{mean-value}, we have
 $$
 F(X)^*+F(X)=\widetilde{\bf B}^{(m)}_{\frac{1}{r}X}\left[F(rW)^*+F(rW)\right].
 $$
 Since the noncommutative Berezin transform is a positive map, we deduce that
 $ F(X)^*+F(X)\geq 0$ for any $X\in {\bf D}^m_{f,rad}(\cH)$ whenever $F(rW)^*+F(rW)\geq 0$ for any $r\in [0,1)$. The converse is obviously  true.
 Putting all these things together we  complete the proof.
\end{proof}

The next result is an analogue of Weierstrass  theorem  for free pluriharmonic functions on the noncommutative domain ${\bf D}_{f, rad}^m(\cH)$.

\begin{theorem} \label{Weier} Let $F_k:{\bf D}_{f, rad}^m(\cH)\to B(\cE)\otimes_{min} B( \cH)$, $k\in \NN$,   be a sequence of free pluriharmonic functions such that, for any  $r\in [0,1)$,
 the sequence $\{F_k(r{W}\}_{k=1}^\infty$
 is convergent in the operator norm topology.
  Then there is a free pluriharmonic  function   $F:{\bf D}_{f, rad}^m(\cH)\to B(\cE)\otimes_{min} B( \cH) $ such that $F_k(r{W})$ converges to $F(r{W})$, as $k\to \infty$,   for
any  $r\in [0,1)$. In particular, $F_k$ converges  to $F$ uniformly on any domain $r{\bf D}_f^m(\cH)$, $r\in [0,1)$.
\end{theorem}
\begin{proof}
Assume that $F_k$ has the representation
$$
F_k(X)= \sum_{|\alpha|\geq 1}  B_{(\alpha)}(k)\otimes X_\alpha^* +
\sum_{|\alpha|\geq 0}  A_{(\alpha)}(k)\otimes X_\alpha, \qquad X\in {\bf D}_{f,rad}^m(\cH),
$$
for some operators $\{A_{(\alpha)}(k)\}_{\alpha\in \FF_n^+}$ and $\{B_{(\alpha)}(k)\}_{\alpha\in
\FF_n^+\backslash \{g_0\}}$, where the convergence is in the operator norm topology.
According to Theorem \ref{Toeplitz}, for any $r\in [0,1)$,  $F_k(rW)$ is  in the operator space
$\boldsymbol{\cA}_\cE({\bf D}_f^m)^* +\boldsymbol{\cA}_\cE({\bf D}_f^m)$.
Define the function  $\varphi:[0,1)\to \overline{\boldsymbol{\cA}_\cE({\bf D}_f^m)^*+\boldsymbol{\cA}_\cE({\bf D}_f^m)}^{\|\cdot\|}$ by setting
\begin{equation}\label{fi-lim}
\varphi(r):=\lim_{k\to \infty} F_k(rW),\qquad r\in [0,1).
\end{equation}
Let $0 \leq r<t<1$ and note that
\begin{equation*}
\begin{split}
\widetilde{\bf B}^{(m)}_{\frac{r}{t}W}[\varphi(t)]&=\lim_{k\to \infty}\widetilde{\bf B}^{(m)}_{\frac{r}{t}W}[F_k(tW)]=\lim_{k\to \infty} F_k(rW)=\varphi(r),
\end{split}
\end{equation*}
where the limits are in the operator norm topology. According to Theorem \ref{mean-value} the map $F:{\bf D}_{f, rad}^m(\cH)\to
 B(\cE)\otimes_{min} B( \cH)$ defined by
 $$
 F(X):=\widetilde{\bf B}^{(m)}_{\frac{1}{r}X}[\varphi(r)],\qquad X\in r{\bf D}_f^m(\cH), r\in (0,1),
$$
  is a
free pluriharmonic function.
 and  $F(rW)=\varphi(r)$ for any $r\in [0,1)$. Using relation \eqref{fi-lim}, we obtain $F(rW)=\lim_{k\to\infty} F_k(rW)$, $r\in [0,1)$.
 Since
 $$
 \sup_{X\in r{\bf D}_f^m(\cH)} \|F(X)-F_k(X)\|= \|F(rW)-F_k(rW)\|,
 $$
 we deduce that $F_k$ converges  to $F$ uniformly on any domain $r{\bf D}_f^m(\cH)$, $r\in [0,1)$.
The proof is complete.
\end{proof}

\begin{corollary}  Let $F_k:{\bf D}_{f, rad}^m(\cH)\to B(\cE)\otimes_{min} B( \cH)$, $k\in \NN$,   be a sequence of free pluriharmonic functions such that
$\{F_k(0)\}$ is a convergent sequence  in the  operator norm topology and
$$F_1\leq F_2\leq\cdots.
$$
Then $F_k$ converges to a free pluriharmonic function on ${\bf D}_{f, rad}^m(\cH)$.
\end{corollary}
\begin{proof} We may assume that $F_1\geq 0$, otherwise we take $G_k:=F_k-F_1$, $k\in \NN$. Due to Harnack type inequality for positive free pluriharmonic functions on ${\bf D}_{f,rad}^m(\cH)$ (see \cite{Po-Bohr2}),  if $k\geq q$, then we have
$$
\|F_k(X)-F_q(X)\|\leq \|F_k(0)-F_q(0)\| \frac{1-r}{1+r}
$$
for any $X\in r{\bf D}_f^m(\cH)$. Since $\{F_k(0)\}$ is a Cauchy  sequence in the operator norm, we deduce that $\{F_k\}$ is a uniformly Cauchy sequence on $r{\bf D}_f^m(\cH)$.  Hence $\{F_k(rW)\}$ is a Cauchy sequence and, therefore, convergent in the operator norm topology. Applying Theorem \ref{Weier}, we find  a free pluriharmonic  function   $F:{\bf D}_{f, rad}^m(\cH)\to B(\cE)\otimes_{min} B( \cH) $ such that $F_k(r{W})$ converges to $F(r{W})$, as $k\to \infty$,   for
any  $r\in [0,1)$. In particular, $F_k$ converges  to $F$ uniformly on any domain $r{\bf D}_f^m(\cH)$, $r\in [0,1)$.
The proof is complete.
\end{proof}

 Let $Har_\cE( {\bf D}_{f, rad}^m)$ denote the set of all free pluriharmonic functions  $F:{\bf D}_{f, rad}^m(\cH)\to B(\cE)\otimes_{min} B( \cH)$.
If $F,G\in Har_\cE( {\bf D}_{f, rad}^m)$ and
$0<r<1$, we define
$$
d_r(F,G):=\|F(r{W})-G(r{W})\|.
$$
If  $\cH$ is
an infinite dimensional Hilbert space,  the noncommutative von Neumann inequality for the $n$-tuples in the domain ${\bf D}_f^m(\cH)$ implies
$$
d_r(F,G)=\sup_{{X}\in r{\bf D}_f^m(\cH)}
\|F({X})-G({X})\|.
$$
 Let $\{r_m\}_{m=1}^\infty$ be an increasing sequence  of positive numbers
  convergent  to $1$.
For any $F,G\in Har_\cE( {\bf D}_{f, rad}^m)$, we define
$$
\rho (F,G):=\sum_{k=1}^\infty \left(\frac{1}{2}\right)^k
\frac{d_{r_k}(F,G)}{1+d_{r_k}(F,G)}.
$$
Using standards arguments,
    one can    show  that $\rho$ is a metric
on $Har_\cE( {\bf D}_{f, rad}^m)$.

\begin{theorem}\label{complete-metric}
$\left(Har_\cE( {\bf D}_{f, rad}^m), \rho\right)$  is a complete metric space.
 \end{theorem}

\begin{proof}
It is easy to see   that  if $\epsilon>0$, then  there exists $\delta>0$
and $N\in \NN$ such that, for any   $F,G\in Har_\cE( {\bf D}_{f, rad}^m)$,
 \ $d_{r_N}(F,G)<\delta\implies \rho(F,G)<\epsilon$.
Conversely, if $\delta>0$ and $N\in \NN$ are fixed, then there is
$\epsilon>0$ such that, for  any $F,G\in Har_\cE( {\bf D}_{f, rad}^m)$,  \ $
\rho(F,G)<\epsilon \implies d_{r_N}(F,G) <\delta.
$

 Let
$\{G_k\}_{k=1}^\infty\subset Har_\cE( {\bf D}_{f, rad}^m)$ be a Cauchy sequence in the
metric $\rho$. A consequence of the remark above  is
that
  $\{G_k(r_N{W})\}_{k=1}^\infty $  is a Cauchy sequence
   in $B(\cE\otimes F^2(H_{n}))$,
   for any
  $N\in \NN$. Consequently, for each $N\in \NN$,
the sequence $\{G_k(r_N{W})\}_{k=1}^\infty$ is
convergent in the operator norm. Using
Theorem \ref{Weier}, we find a free pluriharmonic function
$G\in Har_\cE( {\bf D}_{f, rad}^m)$
   such that
$G_k(r{W})$ converges to $G(r{W})$  for
any $r\in [0,1)$. By the observation made at the beginning
of this proof, we conclude that $\rho(G_k, G)\to 0$, as $k\to\infty$,
which completes the proof.
 \end{proof}

\bigskip

\section{Bounded free pluriharmonic functions and Dirichlet extension problem}

In this section, we characterize the  bounded    pluriharmonic functions on ${\bf D}_{f, rad}^m(\cH)$ as    noncommutative Berezin transforms of  weighted right multi-Toeplitz operators and present a noncommutative version of Dirichlet extension problem.

Let us recall some definitions concerning completely bounded maps
 on operator spaces.
We identify $M_k(B(\cH))$, the set of
$k\times k$ matrices with entries in $B(\cH)$, with
$B( \cH^{(k)})$, where $\cH^{(k)}$ is the direct sum of $k$ copies
of $\cH$.
If $\cX$ is an operator space, i.e., a closed subspace
of $B(\cH)$, we consider $M_k(\cX)$ as a subspace of $M_k(B(\cH))$
with the induced norm.
Let $\cX, \cY$ be operator spaces and $u:\cX\to \cY$ be a linear map. Define
the map
$u_k:M_k(\cX)\to M_k(\cY)$ by
$$
u_k ([x_{ij}]_{k}):=[u(x_{ij})]_{k}.
$$
We say that $u$ is completely bounded   if
$
\|u\|_{cb}:=\sup_{k\ge1}\|u_k\|<\infty.
$
When  $\|u\|_{cb}\leq1$
(resp. $u_k$ is an isometry for any $k\geq1$) then $u$ is completely
contractive (resp.~isometric). We call $u$ completely positive
  if $u_k$ is positive for all $k\geq 1$.
   For more information  on completely bounded (resp.~positive) maps, we refer
 to \cite{Pa-book} and \cite{Pi}.

 A free pluriharmonic
function $G:{\bf D}_{f, rad}^m(\cH)\to B(\cE)\otimes_{min} B(\cH)$ is called bounded if
$$\|G\|:=\sup||G(X)\|<\infty,
$$
where the supremum is taken over all $n$-tuples $X\in {\bf D}_{f, rad}^m(\cH)$ and any Hilbert space $\cH$.
 Due to the noncommutative von Neumann inequality for elements in ${\bf D}_{f, rad}^m(\cH)$,  it is enough to assume, throughout this section,  that the Hilbert space $\cH$  is
separable and infinite dimensional.
Denote by  $Har_\cE^\infty({\bf D}_{f, rad}^m)$  the set of all bounded free
pluriharmonic functions on  ${\bf D}_{f, rad}^m$ with coefficients in
$B(\cE)$, where $\cE$ is a separable Hilbert space.
 For each $k=1,2,\ldots$,
we define the norms $\|\cdot
\|_k:M_k\left(Har_\cE^\infty({\bf D}_{f, rad}^m)\right)\to [0,\infty)$ by
setting
$$
\|[F_{ij}]_k\|_k:= \sup \|[F_{ij}(X)]_k\|,
$$
where the supremum is taken over all $n$-tuples $X\in {\bf D}_{f, rad}^m(\cH)$ and any Hilbert space $\cH$. It is easy to see that the norms
$\|\cdot\|_k$, $k=1,2,\ldots$, determine  an operator space
structure  on $Har_\cE^\infty({\bf D}_{f, rad}^m)$,
 in the sense of Ruan (see e.g. \cite{ER}).

\begin{theorem}\label{bounded}
  If $F:{\bf D}_{f, rad}^m(\cH)\to B(\cE)\otimes_{min}B(\cH)$, then
 the following statements are equivalent:
\begin{enumerate}
\item[(i)] $F$ is a bounded free pluriharmonic function on
${\bf D}_{f, rad}^m(\cH)$;
\item[(ii)]
there exists $\psi\in \overline{\boldsymbol{\cA}_\cE({\bf D}_f^m)^*+\boldsymbol{\cA}_\cE({\bf D}_f^m)}^{SOT}$ such
that \ $F(X)=\widetilde{\bf B}^{(m)}_X[\psi]$ for $X\in {\bf D}_{f, rad}^m(\cH)$,
\end{enumerate}
where $\widetilde{\bf B}^{(m)}_X$ is the noncommutative Berezin transform at $X$.
In this case,
  $\psi=\text{\rm SOT-}\lim\limits_{r\to 1}F(rW).
  $
   Moreover, the map
$$
\Phi:Har^\infty({\bf D}_{f, rad}^m)\to \overline{\boldsymbol{\cA}_\cE({\bf D}_f^m)^*+\boldsymbol{\cA}_\cE({\bf D}_f^m)}^{SOT}\quad
\text{ defined by } \quad \Phi(F):=\psi
$$ is a completely   isometric isomorphism of operator spaces, where $\boldsymbol{\cA}_\cE({\bf D}_f^m):= B(\cE)\otimes_{min}\cA({\bf D}_f^m)$ and  $\cA({\bf D}_f^m)$ is the noncommutative domain
algebra.
\end{theorem}

\begin{proof}
Let $F\in {\bf D}_{f, rad}^m(\cH)$ and note that, due to Proposition \ref{ser-def}, it has a representation
\begin{equation*} F(X)=\sum_{k=1}^\infty \sum_{|\alpha|=k} B_{(\alpha)}\otimes
X_\alpha^* +A_{(0)}\otimes  I+ \sum_{k=1}^\infty
 \sum_{|\alpha|=k} A_{(\alpha)}\otimes  X_\alpha,
\end{equation*}
where the series are convergent in the  operator norm topology
  for any $X\in {\bf D}_{f, rad}^m(\cH)$. Consequently, we have
  $ F(rW)\in \boldsymbol{\cA}_\cE({\bf D}_f^m)^* +\boldsymbol{\cA}_\cE({\bf D}_f^m)$ for any $r\in [0,1)$, and $\sup_{r\in [0,1)}\|F(rW)\|<\infty$.
  Applying Theorem \ref{Toeplitz}, we find a unique weighted right multi-Toeplitz operator $T\in B(\cE\otimes F^2(H_n))$ such that
  \begin{equation}
  \label{TT}
  T=\text{\rm SOT-}\lim_{r\to 1}F(rW) \quad \text{ and }\quad \|T\|=\sup_{r\in [0,1)}\|F(rW)\|.
  \end{equation}
 Therefore,  $T\in \overline{\boldsymbol{\cA}_\cE({\bf D}_f^m)^*+\boldsymbol{\cA}_\cE({\bf D}_f^m)}^{SOT}$. Now, we prove that
  \ $F(X)=\widetilde{\bf B}^{(m)}_X[T]$ for $X\in {\bf D}_{f, rad}^m(\cH)$.
 Indeed,  since $ F(rW)\in \boldsymbol{\cA}_\cE({\bf D}_f^m)^* +\boldsymbol{\cA}_\cE({\bf D}_f^m)$, we have
 $$
 F(rX)=(I_\cE\otimes K_{f,X}^{(m)})^*(F(rW)\otimes I_\cH)(I_\cE\otimes  K_{f,X}^{(m)})
 $$
 for $X\in {\bf D}_{f, rad}^m(\cH)$ and $r\in [0,1)$. Since the map $Y\mapsto Y\otimes I$ is SOT-continuous on bounded sets, we use relation \eqref{TT} to deduce that
  $\text{\rm SOT-}\lim_{r\to 1}F(rX)=\widetilde{\bf B}^{(m)}_X[T]$ for $X\in {\bf D}_{f, rad}^m(\cH)$. On the other hand, since $F$ is continuous on $ {\bf D}_{f, rad}^m(\cH)$ with respect to the operator norm topology, we conclude that $F(X)=\widetilde{\bf B}^{(m)}_X[T]$ for $X\in {\bf D}_{f, rad}^m(\cH)$, which shows that item (ii) holds.
  Conversely, assume that (ii) holds. Then
   $F(X)=\widetilde{\bf B}^{(m)}_X[\psi]$ for any $X\in {\bf D}_{f, rad}^m(\cH)$ and some $\psi\in \overline{\boldsymbol{\cA}_\cE({\bf D}_f^m)^*+\boldsymbol{\cA}_\cE({\bf D}_f^m)}^{SOT}$. Due to Corollary \ref{multi}, $\psi$ is a weighted right multi-Toeplitz operator on $\cE\otimes F^2(H_n)$. Applying Theorem \ref{Toeplitz},  we find be two sequences $\{A_{(\alpha)}\}_{\alpha\in \FF_n^+}$ and $\{B_{(\alpha)}\}_{\alpha\in
\FF_n^+\backslash \{g_0\}}$ of  operators on a Hilbert space $\cE$ such that,
   $\psi=\text{\rm SOT-}\lim_{r\to \infty} G(rW)$,
   where
   $$
   G(rW):=\sum_{k=1}^\infty \sum_{|\alpha|=k}
B_{(\alpha)}\otimes  r^{|\alpha|} W_\alpha^* + A_{(0)}\otimes I
 +\sum_{k=1}^\infty
\sum_{|\alpha|=k} A_{(\alpha)}\otimes r^{|\alpha|} W_\alpha,
$$
with the convergence is in the operator norm topology. Moreover, we have
$\sup_{r\in [0,1)}\|G(rW)\|=\|\psi\|$.
Define the bounded free pluriharmonic function $G:{\bf D}_{f, rad}^m(\cH)\to B(\cE)\otimes_{min}B(\cH)$ by setting
\begin{equation*} G(X)=\sum_{k=1}^\infty \sum_{|\alpha|=k} B_{(\alpha)}\otimes
X_\alpha^* +A_{(0)}\otimes  I+ \sum_{k=1}^\infty
 \sum_{|\alpha|=k} A_{(\alpha)}\otimes  X_\alpha, \qquad, X\in {\bf D}_{f, rad}^m(\cH),
\end{equation*}
where the series are convergent in the  operator norm topology. Note that
\begin{equation*}
\begin{split}
\widetilde{\bf B}^{(m)}_X[\psi]&= \text{SOT-}\lim_{r\to 1}\widetilde{\bf B}^{(m)}_X[G(rW)]\\
&= \text{SOT-}\lim_{r\to 1}\left(\sum_{k=1}^\infty \sum_{|\alpha|=k} r^{|\alpha|}B_{(\alpha)}\otimes
X_\alpha^* +A_{(0)}\otimes  I+ \sum_{k=1}^\infty
 \sum_{|\alpha|=k} A_{(\alpha)}\otimes r^{|\alpha|} X_\alpha\right)\\
 &=\text{SOT-}\lim_{r\to 1}G(rX)=G(X),
\end{split}
\end{equation*}
where the last equality is due to the continuity of $G$. Therefore, $G(X)=F(X)$ for
any $ X\in {\bf D}_{f, rad}^m(\cH)$.

Now, let $[F_{ij}]_{k\times k}$ be a $k\times k$ matrix with entries in $Har_\cE^\infty({\bf D}_{f, rad}^m)$. As in the case when $k=1$, we can use the noncommutative von Neumann inequality for  the domain ${\bf D}_{f}^{m}$,  to show that
\begin{equation*}
\|[F_{ij}]_{k\times k}\|=\sup_{r\in [0,1)}\|[F_{ij}(rW)]_{k\times k}\|
\end{equation*}
and that $T_{ij}:=\text{\rm SOT-}\lim_{r\to 1}F_{ij}(rW)$ are weighted right multi-Toeplitz operators.
Since
$$
(I_\cE\otimes K_{f,rW}^{(m)})^*(T_{ij}\otimes I_\cH)(I_\cE\otimes  K_{f,rW}^{(m)})=F_{ij}(rW),\qquad r\in [0,1),
$$
we deduce that $\|[F_{ij}(rW)]_{k\times k}\|\leq \|[T_{ij}]_{k\times k}\|$, $r\in [0,1)$, which, due to the convergence above, implies $\|[F_{ij}(rW)]_{k\times k}\|= \|[T_{ij}]_{k\times k}\|$. This completes the proof.
\end{proof}

A consequence of Theorem \ref{bounded} and  Corollary \ref{multi} is the following noncommutative version of Herglotz theorem (see \cite{Her}, \cite{H}).
\begin{corollary} Any non-negative bounded free pluriharmonic function on ${\bf D}_{f,rad}^m$ is the Berezin transform of a positive weighted right  multi-Toeplitz operator on $\cE\otimes F^2(H_n)$.

\end{corollary}

\begin{corollary} If $F:{\bf D}_{f, rad}^m(\cH)\to B(\cE)\otimes_{min}B(\cH)$ is  a bounded free pluriharmonic function and  $Y\in {\bf D}_f^m(\cH)$ is a pure $n$-tuple of operators, then $\lim_{r\to 1} F(rY)$ exists in the strong operator topology.
\end{corollary}
\begin{proof}
Assume that $F$ has the representation
\begin{equation*} F(X)=\sum_{k=1}^\infty \sum_{|\alpha|=k} B_{(\alpha)}\otimes
X_\alpha^*  + \sum_{k=0}^\infty
 \sum_{|\alpha|=k} A_{(\alpha)}\otimes  X_\alpha,\qquad X\in {\bf D}_{f, rad}^m(\cH),
\end{equation*}
where the series are convergent in the  operator norm topology. Due to Theorem \ref{bounded}, we find a unique weighted right multi-Toeplitz operator $T\in B(\cE\otimes F^2(H_n))$ such that
  \begin{equation}
  \label{TT2}
  T=\text{\rm SOT-}\lim_{r\to 1}F(rW) \quad \text{ and }\quad \|T\|=\sup_{r\in [0,1)}\|F(rW)\|.
  \end{equation}
Let  $Y\in {\bf D}_f^m(\cH)$ be  a pure $n$-tuple of operators and let $r\in [0,1)$.
Then we have
\begin{equation*}
\begin{split}
F(rY)&=\sum_{k=1}^\infty \sum_{|\alpha|=k} B_{(\alpha)}\otimes
r^{|\alpha|}Y_\alpha^*  + \sum_{k=0}^\infty
 \sum_{|\alpha|=k} A_{(\alpha)}\otimes r^{|\alpha|} Y_\alpha\\
 &={\bf B}_Y^{(m)}[F(rW)]\\
 &=(I_\cE\otimes K_{f,Y}^{(m)})^*(F(rW)\otimes I_\cH)(I_\cE\otimes  K_{f,Y}^{(m)}),
\end{split}
\end{equation*}
where the convergence of the series is in the operator norm topology.
Consequently, since the map $A\mapsto A\otimes I$ is SOT-continuous on bounded sets, relation \eqref{TT2} implies that $\text{\rm SOT-}\lim_{r\to 1} F(rY)$ exists and it is equal to $(I_\cE\otimes K_{f,Y}^{(m)})^*(T\otimes I_\cH)(I_\cE\otimes  K_{f,Y}^{(m)})$. The proof is complete.
\end{proof}

\begin{corollary} \label{mean-bounded} Given a function $F:{\bf D}_{f, rad}^m(\cH)\to B(\cE)\otimes_{min}B(\cH)$, the following statements are equivalent:
\begin{enumerate}
\item[(i)] $F$ is a bounded free plurihamonic function.

\item[(ii)] There is a bounded function $\varphi:[0,1)\to \overline{\boldsymbol{\cA}_\cE({\bf D}_f^m)^*+\boldsymbol{\cA}_\cE({\bf D}_f^m)}^{\|\cdot\|}$  which satisfies the relation
$$
\varphi(r)=\widetilde{\bf B}^{(m)}_{\frac{r}{t}W}[\varphi(t)],\qquad \text{for any }\  0\leq r<t<1,
$$
and $F(X):=\widetilde{\bf B}^{(m)}_{\frac{1}{r}X}[\varphi(r)]$ for any  $X\in r{\bf D}_f^m(\cH)$ and $ r\in (0,1)$.
\end{enumerate}
Moreover, $F$ and $\varphi$ uniquely determine each other and $F(rW)=\varphi(r)$ for any $r\in [0,1)$.
\end{corollary}
\begin{proof}
Assume that $F$ is a bounded free pluriharmonic function and has representation
\begin{equation*} F(X)=\sum_{k=1}^\infty \sum_{|\alpha|=k} B_{(\alpha)}\otimes
X_\alpha^*  + \sum_{k=0}^\infty
 \sum_{|\alpha|=k} A_{(\alpha)}\otimes  X_\alpha,\qquad X\in {\bf D}_{f, rad}^m(\cH),
\end{equation*}
where the series are convergent in the  operator norm topology. Then $\sup_{r\in [0,1)} \|F(rW)\|<\infty$ and
$$
F(rW)=\widetilde{\bf B}_{\frac{r}{t}W}^{(m)}[F(tW)],\qquad 0\leq r<t<1.
$$
Define $\varphi:[0,1)\to \overline{\boldsymbol{\cA}_\cE({\bf D}_f^m)^*+\boldsymbol{\cA}_\cE({\bf D}_f^m)}^{\|\cdot\|}$ by setting $\varphi(r):=F(rW)$.
Note that, if $X\in r{\bf D}_f^m(\cH)$, then
\begin{equation*}
\begin{split}
F(X)&=\lim_{\delta\to 1}\left(\sum_{k=1}^\infty \sum_{|\alpha|=k} B_{(\alpha)}\otimes
\delta^{|\alpha|} X_\alpha^*  + \sum_{k=0}^\infty
 \sum_{|\alpha|=k} A_{(\alpha)}\otimes  \delta^{|\alpha|}X_\alpha\right)\\
 &=\lim_{\delta\to 1} \widetilde{\bf B}^{(m)}_{\frac{\delta}{r}X}[\varphi(r)]=\widetilde{\bf B}^{(m)}_{\frac{1}{r}X}[\varphi(r)].
\end{split}
\end{equation*}
Conversely, assume that item (ii) holds. Applying Theorem \ref{mean-value} to $\varphi$, we deduce that $F$ is a free pluriharmonic function and $F(rW)=\varphi(r)$ for any $r\in [0,1)$. Since $\varphi$ is bounded, we also have $\|F\|\leq \sup_{r\in [0,1)} \|F(rW)\|<\infty$. This completes the proof.
\end{proof}

We denote by $Har_\cE^c({\bf D}_{f, rad}^m)$ the set of all
  free pluriharmonic functions on ${\bf D}_{f, rad}^m(\cH)$ with operator-valued coefficients in $B(\cE)$, which
 have continuous extensions   (in the operator norm topology) to
the domain ${\bf D}_f^m(\cH)$. Here is our noncommutative version of the Dirichlet extension problem for harmonic functions \cite{H}.

\begin{theorem}\label{Dirichlet}  If $F:{\bf D}_{f, rad}^m(\cH)\to B(\cE)\otimes_{min} B( \cH)$, then
 the following statements are equivalent:
\begin{enumerate}
\item[(i)] $F$ is a free pluriharmonic function on ${\bf D}_{f, rad}^m(\cH)$ which
 has a continuous extension  (in the operator norm topology) to
the domain ${\bf D}_f^m(\cH)$;

\item[(ii)] $F$ is a free pluriharmonic function on ${\bf D}_{f, rad}^m(\cH)$
such that \ $F(rW)$ converges in the operator norm
topology, as $r\to 1$.
\item[(iii)]
there exists $\psi\in \overline{\boldsymbol{\cA}_\cE({\bf D}_f^m)^*+\boldsymbol{\cA}_\cE({\bf D}_f^m)}^{\|\cdot\|}$ such that
$F(X)=\widetilde{\bf B}^{(m)}_X[\psi]$ for $X\in {\bf D}_{f, rad}^m(\cH)$;
\end{enumerate}
In this case, $  \psi=\lim\limits_{r\to 1}F(rW),$ where
the convergence is in the operator norm. Moreover, the map $
\Phi:Har_\cE^c({\bf D}_{f, rad}^m)\to
\overline{\boldsymbol{\cA}_\cE({\bf D}_f^m)^*+\boldsymbol{\cA}_\cE({\bf D}_f^m)}^{\|\cdot\|}\quad \text{ defined
by } \quad \Phi(F):=\psi $ is a  completely   isometric isomorphism of
operator spaces.
\end{theorem}

\begin{proof} The implication (i)$\implies$(ii)  is clear. Assume that (ii) holds and note that the function $\varphi :[0,1]\to \overline{\boldsymbol{\cA}_\cE({\bf D}_f^m)^*+\boldsymbol{\cA}_\cE({\bf D}_f^m)}^{\|\cdot\|}$ given by
$\varphi(r):=F(rW)$ if $r\in [0,1)$ and $\varphi(1):=\lim_{r\to 1} F(rW)$ is continuous and bounded. Setting $\psi:=\varphi(1)$ and using Theorem \ref{bounded}, we deduce that
$F(X)=\widetilde{\bf B}^{(m)}_X(\psi)$ for $X\in {\bf D}_{f, rad}^m(\cH)$. Therefore, the implication (ii)$\implies$(iii) holds true. Now, we prove the implication
(iii)$\implies$(i). Assume that item (iii) holds. Thus there exists $\psi\in \overline{\boldsymbol{\cA}_\cE({\bf D}_f^m)^*+\boldsymbol{\cA}_\cE({\bf D}_f^m)}^{\|\cdot\|}$ such that
$F(X)=\widetilde{\bf B}^{(m)}_X(\psi)$ for $X\in {\bf D}_{f, rad}^m(\cH)$.
According to Theorem \ref{bounded}, $F$ is a bounded free pluriharmonic function on
${\bf D}_{f, rad}^m(\cH)$, $\|F\|=\|\psi\|$, and $\psi=\text{\rm SOT-}\lim_{r\to 1}F(rW)$.
In what follows, we show that $\psi= \lim_{r\to 1}F(rW)$ in the operator norm topology.
Indeed, let $\sum_{k=1}^\infty \sum_{|\alpha|=k} B_{(\alpha)}\otimes
W_\alpha^*  + \sum_{k=0}^\infty
 \sum_{|\alpha|=k} A_{(\alpha)}\otimes  W_\alpha$ be the Fourier representation of $\psi$ and note that
 \begin{equation*} F(X)=\sum_{k=1}^\infty \sum_{|\alpha|=k} B_{(\alpha)}\otimes
X_\alpha^*  + \sum_{k=0}^\infty
 \sum_{|\alpha|=k} A_{(\alpha)}\otimes  X_\alpha,\qquad X\in {\bf D}_{f, rad}^m(\cH),
\end{equation*}
where the series are convergent in the  operator norm topology.  Then for any $r\in [0,1)$, $F(rW)\in \boldsymbol{\cA}_\cE({\bf D}_f^m)^*+\boldsymbol{\cA}_\cE({\bf D}_f^m)$
and $F(rW)=\widetilde{\bf B}^{(m)}_{rW}(\psi)$. Since $\psi\in \overline{\boldsymbol{\cA}_\cE({\bf D}_f^m)^*+\boldsymbol{\cA}_\cE({\bf D}_f^m)}^{\|\cdot\|}$,
Theorem \ref{funct-calc} implies that $\psi=\lim_{r\to 1} \widetilde{\bf B}^{(m)}_{rW}(\psi)$ in the operator norm topology. Consequently, $\lim_{r\to 1} F(rW)=\psi$ in the operator norm topology, which proves our assertion.

Let $Y\in {\bf D}_f^m(\cH)$ and define $\widetilde F(Y):=\widetilde{\bf B}^{(m)}_{rY}(\psi)$ We remark that,  due to Theorem \ref{funct-calc},  the latter  limit exists in  the operator norm topology. It remains to prove that $\widetilde F|_{{\bf D}_{f, rad}^m(\cH)}=F$ an $\widetilde F$ is continuous on ${\bf D}_{f }^m(\cH)$.
Indeed, if $X\in {\bf D}_{f, rad}^m(\cH)$, then $X$ is a pure $n$-tuple and Theorem
\ref{funct-calc} implies that $\lim_{r\to 1} \widetilde{\bf B}^{(m)}_{rX}(\psi)= {\bf B}^{(m)}_{X}(\psi)$. Consequently, $\widetilde F(X)=F(X)$ for any $X\in  {\bf D}_{f, rad}^m(\cH)$.
Now, we prove the continuity of $\widetilde F$ on ${\bf D}_f^m(\cH)$.
Since $\psi=\lim_{r\to 1}F(rW)$ in the operator norm topology, for any $\epsilon>0$ there exists $r_0\in (0,1)$ such that $\|\psi-F(r_0W)\|<\frac{\epsilon}{3}$.
Since $\psi-F(r_0W)\in \overline{\boldsymbol{\cA}_\cE({\bf D}_f^m)^*+\boldsymbol{\cA}_\cE({\bf D}_f^m)}^{\|\cdot\|}$, we can use again Theorem \ref{funct-calc} and deduce that
$$
\lim_{r\to 1} \widetilde{\bf B}^{(m)}_{rY}[\psi-F(r_0W)]=\widetilde F(Y)-F(r_0Y)
$$
and
$$
\|\widetilde F(Y)-F(r_0Y)\|\leq \|\psi-F(r_0W)\|<\frac{\epsilon}{3}, \qquad Y\in {\bf D}_f^m(\cH).
$$
On the other hand, since $F$ is continuous on ${\bf D}_{f, rad}^m(\cH)$, there is $\delta>0$ such that $\|F(r_0Y)-F(r_0 Z)\|<\frac{\epsilon}{3}$
for any $Z\in {\bf D}_f^m(\cH)$ such that $\|Y-Z\|<\delta$.
Using the estimations above, we note that
\begin{equation*}
\begin{split}
\|\widetilde F(Y)-\widetilde F(Z)\|\leq \|\widetilde F(Y)-F(r_0Y)\|+\|F(r_0Y)-F(r_0Z)\|
+\|F(r_0Z)-\widetilde F(Z)\|<\epsilon
\end{split}
\end{equation*}
for any $Y,Z\in {\bf D}_f^m(\cH)$ such that $\|Y-Z\|<\delta$.
The last part of the theorem follows from Theorem \ref{bounded}.
The proof is complete.
\end{proof}

\begin{corollary}  Given a function $F:{\bf D}_{f, rad}^m(\cH)\to B(\cE)\otimes_{min}B(\cH)$, the following statements are equivalent:
\begin{enumerate}
\item[(i)] $F$ is a  free plurihamonic function which has continuous extension to ${\bf D}_f^m(\cH)$.

\item[(ii)] There is a continuous function $\varphi:[0,1]\to \overline{\boldsymbol{\cA}_\cE({\bf D}_f^m)^*+\boldsymbol{\cA}_\cE({\bf D}_f^m)}^{\|\cdot\|}$  in the operator norm topology  which satisfies the relation
$$
\varphi(r)=\widetilde{\bf B}^{(m)}_{\frac{r}{t}W}[\varphi(t)],\qquad \text{for any }\  0\leq r<t<1,
$$
and $F(X):=\widetilde{\bf B}^{(m)}_{\frac{1}{r}X}[\varphi(r)]$ for any  $X\in r{\bf D}_f^m(\cH)$ and $ r\in (0,1)$.
\end{enumerate}
Moreover, $F$ and $\varphi$ uniquely determine each other and $F(rW)=\varphi(r)$ for any $r\in [0,1)$.
\end{corollary}
\begin{proof} The proof is similar to that of Corollary \ref{mean-bounded}, but uses Theorem \ref{Dirichlet}. We leave it to the reader.
\end{proof}

\bigskip

\section{Cauchy transforms and  functional
calculus for noncommuting  operators} \label{Cauchy}

In this section, we use noncommutative Cauchy transforms associated with the domain ${\bf D}_f^m(\cH)$, to  provide a free analytic functional calculus for $n$-tuples of operators $X=(X_1,\ldots, X_n)\in B(\cH)^n$ with the spectral radius of the reconstruction operator  strictly less than 1. This extends to free pluriharmonic functions and has several consequences.

Let $f=\sum_{\alpha\in \FF_n^+} a_\alpha Z_\alpha$, $\alpha\in \CC$,
be a positive regular free holomorphic. For any
$n$-tuple of operators $X:=(X_1,\ldots, X_n)\in B(\cH)^n$
 such  that
 $ \sum\limits_{|\alpha|\geq 1}
a_{\alpha} X_{\alpha}X_{\alpha}^*$ is SOT-convergent, we define the
joint spectral radius of $X$  with respect to the noncommutative
domain ${\bf D}_f^{m}$ to be
$$
r_f(X):=\lim_{k\to\infty}\|\Phi_{f,X}^k(I)\|^{1/2k},
$$
where the positive linear map $\Phi_{f,X}:B(\cH)\to B(\cH)$ is given
by
$$
\Phi_{f,X}(Y):=\sum_{|\alpha|\geq 1} a_\alpha X_\alpha
YX_\alpha^*,\quad Y\in B(\cK),
$$
and  the convergence is in the week operator topology.
In the particular case when $f:=Z_1+\cdots +Z_n$, we obtain the
usual definition of the joint operator radius for $n$-tuples of
operators.

Since $\sum\limits_{|\alpha|\geq 1} a_{\tilde \alpha} \Lambda_\alpha
\Lambda_\alpha^*$ is SOT convergent, one can easily see that the  series
$\sum\limits_{|\alpha|\geq 1} a_{\tilde\alpha} \Lambda_\alpha
\otimes X_{\tilde\alpha}^*$ is SOT-convergent  in $B(F^2(H_n)\otimes
\cH)$.
We call the operator
$$ R_{\tilde f,X}:=\sum_{|\alpha|\geq 1} a_{\tilde\alpha} \Lambda_\alpha \otimes
X_{\tilde\alpha}^*$$
 the {\it  reconstruction operator} associated with the $n$-tuple
  $X:=(X_1,\ldots, X_n)$ and the noncommutative domain ${\bf D}_f^{m}$.
 Note  that
\begin{equation*}
 \left\|R_{\tilde f,X}^k\right\|\leq
\left\|\Phi^k_{\tilde f,\Lambda}(I)\right\|^{1/2}
\left\|\Phi_{f,X}^k(I)\right\|^{1/2},\quad k\in \NN,
\end{equation*}
where $\tilde f:=\sum\limits_{|\alpha|\geq 1} a_{\tilde
\alpha}Z_\alpha$ and $\Phi_{\tilde
f,\Lambda}(Y):=\sum\limits_{|\alpha|\geq 1}
a_{\tilde\alpha}\Lambda_\alpha Y \Lambda_\alpha^*$. Consequently, we deduce that
that
\begin{equation*}
r\left(R_{\tilde f,X}\right)\leq r_{\tilde f}(\Lambda)r_f(X),
\end{equation*}
where $r(A)$ denotes the usual spectral radius of an operator $A$.
Since
$\left\|\Phi_{\tilde f,\Lambda}(I)\right\|\leq 1$ (see relation \eqref{tild-Lamb}), we deduce that
$r_{\tilde f}(\Lambda)\leq 1$. This implies
\begin{equation*}
\label{r<r} r(R_{\tilde f,X})\leq r_f(X).
\end{equation*}

Assume now that $X:=(X_1,\ldots, X_n)\in B(\cH)^n$ is an $n$-tuple
of operators with  $r(R_{\tilde f,X})<1$. Note that the latter condition holds if $r_f(X)<1$.
 We introduce the {\it Cauchy kernel} associated  with   $X$  to be
  the operator
\begin{equation*}
  C_{f,X}^{(m)}:=\left(
I-R_{\tilde f,X} \right)^{-m},
\end{equation*}
which is well-defined
and
\begin{equation*}
 C_{f,X}^{(m)}=\left(\sum_{k=0}^\infty R_{\tilde f,X}^k\right)^m,
\end{equation*}
where the convergence is in the operator norm topology.

We remark that   $C_{f,X}^{(m)}\in R^\infty({\bf D}_f^m)\bar\otimes B(\cH)$, the $WOT$-closed operator
    algebra generated by the spatial tensor product.
    Moreover,
    its  Fourier representation
    is
    \begin{equation}
    \label{Fourier-Cauc}
    C_{f,X}^{(m)}=\sum_{\beta\in \FF_n^+} \Lambda_\beta\otimes b^{(m)}_{\tilde
    \beta}X_{\tilde \beta}^*,
    \end{equation}
where the coefficients $b_\alpha^{(m)}$, $\alpha\in \FF_n^+$ are given by
relation   \eqref{b-al}. In the particular case when
    $f$ is a polynomial,   the Cauchy kernel is
    in $\cR({\bf D}_f^m)\bar\otimes_{min} B(\cH)$.

 Given an $n$-tuple of operators  $X:=(X_1,\ldots, X_n)\in B(\cH)^n$
  with  $r(R_{\tilde f,X})<1$,
  we define the {\it Cauchy transform} at
  $X$ to be the mapping
$$
\cC^{(m)}_{f,X}:B(F^2(H_n))\to B(\cH)
$$
defined by
$$
\left< \cC_{f,X}^{(m)}(A)x,y\right>:= \left<(A\otimes I_\cH)(1\otimes x),
C_{f,X}^{(m)}(1\otimes y)\right>, \qquad x,y\in \cH.
$$
 The operator $\cC_{f,X}^{(m)}(A)$ is called the Cauchy
transform  of $A$ at $X$.

In what follows,  we provide a {\it free analytic functional calculus} for
$n$-tuples of operators $X\in B(\cH)^n$
  with  $r(R_{\tilde f,X})<1$.

\begin{theorem}\label{f-cal} Let $p\in \CC\left<Z_1,\ldots, Z_n\right>$ be a positive regular
noncommutative polynomial  and  let $X:=(X_1,\ldots, X_n)\in B(\cH)^n$
  be an $n$-tuple of operators with  with  $r(R_{\tilde p,X})<1$.
  If \ $$G:=\sum_{|\alpha|\geq 1}d_\alpha Z_\alpha^*+\sum\limits_{\alpha\in \FF_n^+}
c_\alpha Z_\alpha$$
 is  a free pluriharmonic  function on the
noncommutative  domain ${\bf D}_p^m(\cH)$, then
$$
G(X):=\sum\limits_{s=1}^\infty
\sum\limits_{|\alpha|=s} d_\alpha X_\alpha^*+\sum\limits_{s=0}^\infty
\sum\limits_{|\alpha|=s} c_\alpha X_\alpha
$$ is convergent in the operator norm  of $B(\cH)$
and the map
 $$\Psi_{p,X}: (Har({\bf D}_{p, rad}^m), \rho) \to (B(\cH), \|\cdot \|)\quad \text{ defined by } \quad
\Psi_{p,X}(G):=G(X)
$$
is a continuous. In particular, $\Psi_{p,X}|_{Hol({\bf D}_{p, rad}^m)}$  is a continuous unital algebra homomorphism. Moreover, the free
analytic functional calculus on $Hol({\bf D}_{p,rad}^m)$ is uniquely determined by the map
$$
Z_i\mapsto X_i,\qquad i\in \{1,\ldots,n\}.
$$
\end{theorem}

\begin{proof}  Note that,
 using  relations \eqref{Fourier-Cauc}, \eqref{WbWb}, \eqref{WbWb-r},
 we obtain
\begin{equation*}
\begin{split}
\left<\cC_{p,X}^{(m)}(W_\alpha)x,y\right>&= \left< (W_\alpha\otimes
I_\cH)(1\otimes x), C_{p,X}^{(m)}
(1\otimes y)\right>\\
 &=\left< \frac{1}{\sqrt b_\alpha^{(m)}}e_\alpha\otimes x, \sum_{\beta\in
\FF_n^+}
b_{\tilde \beta}^{(m)}(\Lambda_\beta\otimes X_{\tilde \beta}^*)(1\otimes y)\right>\\
&= \left< \frac{1}{\sqrt b_\alpha^{(m)}}e_\alpha\otimes x, \sum_{\beta\in
\FF_n^+} \sqrt {b_{\tilde \beta}^{(m)}}e_{\tilde \beta}\otimes
X_{\tilde \beta}^*y\right>\\
&=\left<X_\alpha x,y\right>
\end{split}
\end{equation*}
for any $ x,y\in \cH$. Hence we deduce that, for any polynomial
$q\in \CC\left<Z_1,\ldots, Z_n\right>$,

\begin{equation*}
\left< q(X)x,y\right>=\left< (q(W)\otimes I_\cH)(1\otimes x), C_{p,X}^{(m)}(1\otimes y)\right>
\end{equation*}
and
\begin{equation}\label{ine-bound}
\|q(X)\|\leq \|q(W)\|\|C_{p,X}^{(m)}\|.
\end{equation}
 Since $F
:=\sum\limits_{\alpha\in \FF_n^+} c_\alpha   Z_\alpha $ is a free
holomorphic function on ${\bf D}_{p, rad}^m$,  the series $F(rW):=\sum\limits_{k=0}^\infty \sum\limits_{|\alpha|=k}c_\alpha
r^{|\alpha|} W_\alpha$, $r\in [0,1)$,  converges in the operator norm topology. Now,
using relation \eqref{ine-bound}, we deduce that $F(rX):=\sum\limits_{s=0}^\infty \sum\limits_{|\alpha|=s}c_\alpha
r^{|\alpha|} X_\alpha$ converges in the operator norm topology of
$B(\cH)$,
\begin{equation}
\label{gr-ine}
 \|F(rX)\|\leq \|F(rW)\|\|C_{p,X}^{(m)}\|,
\end{equation}
and
\begin{equation}\label{GGC}
 \left< F(rX)x,y\right>=\left< (F(rW)\otimes I_\cH)(1\otimes x),
C_{p,X}^{(m)}(1\otimes y)\right>
\end{equation}
for any $x,y\in \cH$ and $r\in [0,1)$.

In what follows, we prove that if $r(R_{\tilde p,X})<1$, then there is $t>1$ suxh that $r(R_{\tilde p,tX})<1$. Indeed, since the spectrum of an operator is upper continuous, so is the spectral radius. Consequently, for any $\delta>0$, there is $\epsilon>0$ such that if $\|X-tX\|<\epsilon$, then $r(R_{\tilde p,tX})<r(R_{\tilde p,X})+\delta$. Hence, using the fact that $r(R_{\tilde p,X})<1$, we deduce that there is $t>1$ such that $r(R_{\tilde p,tX})<1$.
Using relations \eqref{gr-ine} and \eqref{GGC} in the particular case when $r=\frac{1}{t}$ and when $X$ is replaced by $tX$, we deduce that
$F(X):=\sum\limits_{k=0}^\infty \sum\limits_{|\alpha|=k}c_\alpha
 X_\alpha$ is convergent in the operator norm topology and
 \begin{equation}
 \label{FFC}
 \left<F(X)x,y\right>=\left<(F\left(\frac{1}{t}W\right)\otimes I_\cH)(1\otimes x), C_{p,tX}^{(m)}(1\otimes y)\right>,\qquad x,y\in \cH.
 \end{equation}
Hence, we obtain
\begin{equation*}
\|F(X)\|\leq\left\|F\left(\frac{1}{t}W\right)\right\|\|C_{p,tX}^{(m)}\|.
\end{equation*}
Similar results hold true for the free holomorphic function $E
:=\sum\limits_{\alpha\in \FF_n^+} \overline{d}_\alpha   Z_\alpha$.
Combining the results, we deduce that
$$
G(X):=\sum\limits_{s=1}^\infty
\sum\limits_{|\alpha|=s} d_\alpha X_\alpha^*+\sum\limits_{s=1}^\infty
\sum\limits_{|\alpha|=s} c_\alpha X_\alpha
$$ is convergent in the operator norm  of $B(\cH)$
and
\begin{equation}
\label{norm-ine}
\|G(X)\|\leq\left( \left\|E\left(\frac{1}{t}W\right)\right\| + \left\|F\left(\frac{1}{t}W\right)\right\|\right)\|C_{p,tX}^{(m)}\|.
\end{equation}

 To
prove the continuity of $\Psi_{p,X}$, let $G_k$ and $G$ be in
$Har({\bf D}_{p, rad}^m)$ such that $G_k\to G$, as $m\to \infty$,  in the metric
$\rho$ of $Har({\bf D}_{p, rad}^m)$.  This is equivalent to the fact that, for
each $r\in [0,1)$,
\begin{equation*}
G_k(rW)\to G(rW),\quad \text{ as }\
k\to\infty,
\end{equation*}
where the convergence is  in the operator norm of $B(F^2(H_n))$.
Employing  relation \eqref{norm-ine}, when $G$ is replaced by $G_k-G$,  we deduce that
\begin{equation*}
\|G_k(X)-G(X)\|\to 0,
\quad \text{ as }\ k\to\infty.
\end{equation*}
which proves the continuity of $\Psi_{p,X}$.

Let $F_j:=\sum\limits_{s=0}^\infty \sum\limits_{|\alpha|=s}c_\alpha^{(j)}
 Z_\alpha$, $j\in \{1,2\}$, be free holomorphic functions on ${\bf D}_{f,rad}^m$.
Recall that $\cA({\bf D}_f^m)$ is the noncommutative domain algebra and $F_1(rW)F_2(rW)=(F_1F_2)(rW)$ for any $r\in [0,1)$.
Setting $p_{j,k}:=\sum\limits_{s=0}^k \sum\limits_{|\alpha|=s}c_\alpha^{(j)}
 Z_\alpha$, we have $p_{j,k}(X)\to F_j(X)$, as $k\to\infty$, in the operator norm for any $X\in {\bf D}_{p,rad}^m(\cH)$.
 Using relation \eqref{FFC}, we obtain
 $$
 \left<p_{1,k}(X)p_{2,k}(X)x,y\right>=\left<(p_{1,k}\left(\frac{1}{t}W\right)
 (p_{2,k}\left(\frac{1}{t}W\right)\otimes I_\cH)(1\otimes x), C_{p,tX}^{(m)}(1\otimes y)\right>,\qquad x,y\in \cH.
$$
Passing to the limit as  $k\to \infty$ and using again relation \eqref{FFC}, we obtain
\begin{equation*}
\begin{split}
\left<F_{1}(X)F_{2}(X)x,y\right>&=\left<(F_{1}\left(\frac{1}{t}W\right)
 (F_{2}\left(\frac{1}{t}W\right)\otimes I_\cH)(1\otimes x), C_{p,tX}^{(m)}(1\otimes y)\right>\\
 &=\left<(F_{1}F_2)\left(\frac{1}{t}W\right)
  \otimes I_\cH)(1\otimes x), C_{p,tX}^{(m)}(1\otimes y)\right>\\
  &=\left<(F_{1}F_2)(X)x,y\right>
\end{split}
\end{equation*}
for any $x,y\in \cH$. Consequently, $\Psi_{p,X}|_{Hol({\bf D}_{p, rad}^m)}$  is a unital algebra homomorphism.

To prove the uniqueness of the free analytic functional calculus,
assume that $\Phi:Hol({\bf D}_{p,rad}^m)\to B(\cH)$ is a continuous unital
algebra homomorphism such that $\Phi(Z_i)=T_i$, \ $i=1,\ldots, n$.
It is clear  that
\begin{equation}\label{pol2}
\Psi_{p,X}(q)=\Phi(q)
\end{equation}
for any polynomial $q\in \CC\left<Z_1,\ldots, Z_n\right>$.   Let
$F=\sum_{s=0}^\infty \sum_{|\alpha|=s} c_\alpha Z_\alpha$ be an
element in $Hol({\bf D}_{p,rad}^m)$ and let $Q_k:=\sum_{s=0}^k\sum_{|\alpha|=s}
c_\alpha Z_\alpha$, \ $k\in \NN$. Since
$$
F(rW)=\sum_{s=0}^\infty \sum_{|\alpha|=s} r^s
c_\alpha W_\alpha$$ and the series $\sum_{s=0}^\infty
r^s\left\|\sum_{|\alpha|=s} c_\alpha W_\alpha\right\|$ converges, we
deduce that
$
Q_k(rW)\to F(rW)
$
in the operator norm, as $k\to\infty$, which shows that $Q_k\to F$ in the
metric $\rho$ of $Hol({\bf D}_{p, rad}^m)$. Hence, using relation \eqref{pol2} and the
continuity of $\Phi$ and $\Psi_{p,T}$, we deduce that
$\Phi=\Psi_{p,T}$.
 This completes the proof.
 \end{proof}

\begin{corollary}\label{F-infty-cauc} Let $X:=(X_1,\ldots, X_n)\in B(\cH)^n$
  be an $n$-tuple of operators with  with  $r(R_{\tilde p,X})<1$ and  let $F\in
Hol({\bf D}_{p, rad}^m)$. If $t>1$ is such that $r(R_{\tilde p,tX})<1$, then
$$
F(X)=\cC_{p,tX}^{(m)}\left[F\left(\frac{1}{t}W\right)\right],
$$
where $F(X)$ is defined by the free analytic
functional calculus.
If, in addition, $F$ is bounded, then

$$
F(X)=\cC_{p,X}^{(m)}(\widetilde F),\qquad \text{ where  }\ \widetilde F=\text{\rm SOT-}\lim_{r\to 1}F(rW).
$$

\end{corollary}

\begin{proof} The first part of the corollary is due to Theorem \ref{f-cal}  (see relation \eqref{FFC}).
Now, we assume that  $F$ is bounded and has the representation  $F:=\sum\limits_{k=0}^\infty\sum\limits_{|\alpha|=k}
c_\alpha Z_\alpha$.     Then,
 we have
$$
F(rW):=\lim_{k\to\infty}\sum_{s=0}^k r^s \sum_{|\alpha|=s} c_\alpha
W_\alpha,\qquad  0<r<1,
$$
in the operator norm of $B(F^2(H_n))$, and
$$
\lim_{k\to\infty}\sum_{s=0}^k r^s \sum_{|\alpha|=s} c_\alpha
X_\alpha=F(rX)
$$
in the operator norm of $B(\cH)$. Now,  due to the continuity of the
noncommutative Cauchy transform in the operator norm, we deduce that
\begin{equation*}
F(rX)=\cC_{p,X}(F(rW)), \qquad r\in [0,1).
\end{equation*}
Since $F$ is bounded, we know that
$\widetilde F:= \lim\limits_{r\to 1} F(rW)$ exists in
the strong operator topology. Since $\|F(rW)\|\leq
 \|\widetilde F\|$, $r\in [0,1)$,  we deduce that
$$\text{\rm SOT-}\lim\limits_{r\to 1}[F(rW)\otimes
I_\cH]=\widetilde F\otimes I_\cH.
$$
According to the proof of Theorem \ref{f-cal}, we have
\begin{equation*}
\|F(X)\|\leq\left\|F\left(\frac{1}{t}W\right)\right\|\|C_{p,tX}^{(m)}\|.
\end{equation*}
Using  this relation, we deduce that
\begin{equation*}
\begin{split}
 \|F(X)-F(\delta X)\|
\leq \left\|F\left(\frac{1}{t}W\right)-F\left(\frac{\delta}{t}W\right) \right\|\|C_{p,tX}^{(m)}\|, \qquad \delta\in (0,1).
\end{split}
\end{equation*}
Since $\left\|F\left(\frac{1}{t}W\right)-F\left(\frac{\delta}{t}W\right) \right\|\to 0$, as
$\delta\to 1$, we obtain   $\lim\limits_{\delta\to 1}
\|F(X)-F(\delta X)\|=0$.
On the other hand, since
\begin{equation*}
\begin{split}
 \left< F(rX)x,y\right>=\left< (F(rW)\otimes I_\cH)(1\otimes x),
C_{p,X}^{(m)}(1\otimes y)\right>
\end{split}
\end{equation*}
for any $x,y\in \cH$ and $r\in [0,1)$, we can pass  to the limit as $r\to 1$ and   obtain $F(X)=\cC_{p,X}^{(m)}[\widetilde F]$. This
  completes the proof.
   \end{proof}

\begin{corollary}\label{conv-u-w*}
Let $X:=(X_1,\ldots, X_n)\in B(\cH)^n$
  be an $n$-tuple of operators with  with  $r(R_{\tilde p,X})<1$.
\begin{enumerate}
\item[(i)] If $\{G_k\}_{k=1}^\infty$ and $G$ are free pluriharmonic
functions  in $Har({\bf D}_{p, rad}^m)$ such that $\|G_k-G\|_\infty\to 0$, as
$k\to \infty$, then $G_k(X)\to G(X)$ in
the operator norm of $B(\cH)$.
\item[(ii)] Let $\{G_k\}_{k=1}^\infty$ and $G$  be bounded free holomorphic functions on ${\bf D}_{p,rad}^m$ and let  $\{\widetilde G_k\}_{k=1}^\infty$ and $\widetilde G$ be the corresponding boundary operators   in  the noncommutative Hardy  algebra
$F^\infty ({\bf D}_p^m)$. If
$\widetilde G_k\to \widetilde G$ in the $w^*$-topology (or strong operator topology)  and
$\|G_k\|_\infty\leq M$ for any $k\in \NN$, then
$G_k(X)\to G(X)$ in the weak operator
topology.
\end{enumerate}
\end{corollary}

Using Theorem \ref{f-cal} one can deduce   the following.

\begin{corollary}
For  any  $n$-tuple of operators  $(X_1,\ldots, X_n)\in
{\bf D}_{p,rad}^m(\cH)$, the free analytic functional calculus  coincides
with the $F^\infty({\bf D}_p^{m})$-functional calculus $($see \cite{Po-domains-models}$)$.
\end{corollary}

Using  Corollary
\ref{F-infty-cauc}, we can obtain the following.

\begin{corollary}\label{an=cauch}
Let $X:=(X_1,\ldots, X_n)\in B(\cH)^n$
  be an $n$-tuple of operators with  with  $r(R_{\tilde p,X})<1$.
Then,  the map $\Psi_{p,X}:F^\infty({\bf D}_p^m)\to B(\cH)$ defined by
$$
\Psi_{p,X}(\widetilde G):=\cC_{p,X}^{(m)}[\widetilde G],
$$
for any $\widetilde G\in F_n^\infty({\bf D}_p^m)$,  is a unital  WOT
continuous  homomorphism such that $\Psi_{f,X}(W_\alpha)=X_\alpha$
for any $\alpha\in \FF_n^+$. Moreover,
$$
\|\Psi_{p,X}(\widetilde G)\|\leq \left(\sum\limits_{k=0}^\infty\left\|R_{\tilde f,X}^k\right\|\right)^m \|\widetilde G\|.
$$
\end{corollary}

\begin{definition}
Let $H_1$ and $H_2$ be two self-adjoint free pluriharmonic functions on ${\bf D}_{f, rad}^m$ with scalar coefficients. We say that $H_2$ is the pluriharmonic conjugate of $H_1$, if $H_1+iH_2$ is a free holomorphic function on  ${\bf D}_{f, rad}^m$.
\end{definition}

\begin{proposition} The free pluriharmonic conjugate of a self-adjoint free pluriharmonic function on ${\bf D}_{f, rad}^m$ is unique up to an additive real constant.
\end{proposition}
\begin{proof}
Assume that $G=\Re F$ with $F\in Hol({\bf D}_{f, rad}^m)$, and let $H$ be a self-adjoint  free pluriharmonic function such that $G+iH=\Lambda\in Hol({\bf D}_{f, rad}^m)$.
Then $H=\frac{2\Lambda-F-F^*}{2i}$ and the equality $H=H^*$ implies $\Re(\Lambda-F)=0$.
Consequently, due to the remarks following Definition \ref{def-pluri}, we have $\Lambda-F=\lambda$ with $\lambda$ is an imaginary complex number. Now, it is clear that
$H=\frac{F-F^*}{2i}-i\lambda$.
The proof is complete.
\end{proof}

\begin{theorem} Let $X:=(X_1,\ldots, X_n)\in B(\cH)^n$
  be an $n$-tuple of operators with  with  $r(R_{\tilde p,X})<1$ and let $F\in Hol({\bf D}_{p,rad}^m)$ be such that   $F(0)$ is real. If $G=\Re F$ and  $t>1$ is such that $r(R_{\tilde p,tX})<1$, then
  $$
  \left< F(X)x,y\right>=\left<(G\left(\frac{1}{t}W\right)\otimes I_\cH)(1\otimes x), \left[2C_{p,tX}^{(m)}-I\right](1\otimes y)\right>,\qquad x,y\in \cH.
  $$
 If, in addition, $F$ is bounded, then
$$
  \left< F(X)x,y\right>=\left<(\widetilde G\otimes I_\cH)(1\otimes x), \left[2C_{p,X}^{(m)}-I\right](1\otimes y)\right>,\qquad x,y\in \cH,
  $$
  where $\widetilde G:=\text{\rm SOT-}\lim_{r\to 1}G(rW)$.
\end{theorem}

\begin{proof} Using the proof of Theorem \ref{f-cal}, we deduce that
\begin{equation*}
\begin{split}
&\left<(F\left(\frac{1}{t}W\right)\otimes I_\cH)(1\otimes x), \left[2C_{p,tX}^{(m)}-I\right](1\otimes y)\right>\\
&\qquad=2\left<(F\left(\frac{1}{t}W\right)\otimes I_\cH)(1\otimes x), C_{p,tX}^{(m)}(1\otimes y)\right>-\left<F\left(\frac{1}{t}W\right)\otimes I_\cH)(1\otimes x),1\otimes y\right>\\
&\qquad = 2\left<F(X)x,y\right>-F(0)\left<x,y\right>
\end{split}
\end{equation*}
and
\begin{equation*}
\begin{split}
&\left<(F\left(\frac{1}{t}W\right)^*\otimes I_\cH)(1\otimes x), \left[2C_{p,tX}^{(m)}-I\right](1\otimes y)\right>\\
&\qquad=
\left<(\overline{F(0)}\otimes I_\cH)(1\otimes x), \left[2C_{p,tX}^{(m)}-I\right](1\otimes y)\right>\\
&\qquad=\overline{F(0)}\left<x,y\right>.
 \end{split}
\end{equation*}
Taking into account that $F(0)\in \RR$ and adding the relations above, we obtain
$$
2\left<F(X) x,y\right>=\left<\left[\left(F\left(\frac{1}{t}W\right)^*+F\left(\frac{1}{t}W\right)\right)\otimes I_\cH\right](1\otimes x), \left[2C_{p,tX}^{(m)}-I\right](1\otimes y)\right>,
$$
which proves the first part of the theorem. In a similar manner, but using Corollary \ref{F-infty-cauc}, one can prove the second part of the theorem.
\end{proof}

We remark that the free pluriharmonic conjugate $H$  of $G$ an be expressed  in terms of $G$, due to the fact that $H=\frac{F-F^*}{2i}-i\lambda$, where $\lambda$ is an imaginary complex number.

      \bigskip

       %


\begin{thebibliography}{99}



\bibitem{Ag2} {\sc J.~Agler},
Hypercontractions and subnormality, {\it J. Operator Theory} {\bf
13} (1985), 203--217.














 %




















\bibitem{BS} {\sc A.~Bottcher and B.~Silbermann}, {\it Analysis of Toeplitz operators}, Springer-Verlag, Berlin, 1990.

\bibitem{BH}{\sc A.~Brown  and  P.R.~Halmos},  Algebraic properties of Toeplitz operators, {\it J. Reine Angew. Math.} {\bf 213} 1963/1964 89--102.

















\bibitem{DP2} {\sc K.~R.~Davidson and D.~Pitts},
The algebraic structure of non-commutative analytic Toeplitz algebras,
{\it  Math. Ann.}
   {\bf 311} (1998),  275--303.


\bibitem{DKP}  {\sc K.R.~Davidson, E.~Katsoulis, and D.~Pitts},
  The structure of free semigroup algebras,
 {\it J. Reine Angew. Math.}
  {\bf 533} (2001), 99--125.



\bibitem{DLP} {\sc K.~R.~Davidson, J.~Li, and D.R.~Pitts},
  Absolutely
continuous representations and a Kaplansky density theorem for free
semigroup algebras, {\it J. Funct. Anal.} {\bf 224} (2005), no. 1,
160--191.

\bibitem{Dou} {\sc   R. G.~Douglas}, {\em Banach algebra techniques in operator theory}, Second edition. Graduate Texts in Mathematics, {\bf 179}, Springer-Verlag, New York, 1998. xvi+194 pp.
















\bibitem{ER} {\sc E.G.~Effros and Z.J.~Ruan},
  {\em Operator spaces},
 London Mathematical Society Monographs. New Series, {\bf 23}.
 The Clarendon Press, Oxford University Press, New York, 2000.





\bibitem{EL} {\sc J.~Eschmeier and S.~Langend\" orfer}, Toeplitz operators with pluriharmonic symbol, preprint 2017.




\bibitem{HW} {\sc P.~Hartman and A.~Wintner}, The spectra of Toeplitz's matrices, {\it Amer. J. Math.}  {\bf 76} (1954), 867-882.

\bibitem{HKZ} {\sc H.~Hedenmalm, B.~ Korenblum, and K.~Zhu},  {\it Theory of Bergman spaces}, Graduate Texts in Mathematics, {\bf 199}, Springer-Verlag, New York, 2000. x+286 pp.

\bibitem{Her} {\sc G.~Herglotz},
\" Uber Potenzreien mit positiven, reelen Teil im Einheitkreis,
Berichte \" uber die Verhaundlungen der k\" oniglich s\" achsischen
Gesellschaft der Wissenschaften zu Leipzig, {\it Math.-Phys. Klasse}
{\bf 63} (1911), 501--511.



\bibitem{H} {\sc K.~Hoffman},  {\em Banach Spaces of Analytic Functions},
Englewood Cliffs: Prentice-Hall, 1962.







    \bibitem{Ken1} {\sc M.~ Kennedy}, Wandering vectors and the reflexivity of free semigroup algebras, {\it J. Reine Angew. Math.} {\bf 653} (2011), 47--73.

\bibitem{Ken2} {\sc M.~ Kennedy}, The structure of an isometric tuple, {\it  Proc. Lond. Math. Soc.} (3) {\bf 106} (2013), no. 5, 1157--1177.



     \bibitem{LO} {\sc I.~Louhichi and A.~Olofsson}, Characterizations of Bergman space Toeplitz operators with harmonic symbols, {\it J.Reine Angew. Math.} {\bf 617} (2008), 1--26.




\bibitem{O1} {\sc A.~Olofsson},
A characteristic operator function for the class of
$n$-hypercontractions, {\it  J. Funct. Anal.} {\bf 236} (2006), no.
2, 517--545.


\bibitem{O2} {\sc A.~Olofsson},  An operator-valued Berezin transform and
the class of $n$-hypercontractions,  {\it Integral Equations Operator Theory}  {\bf 58}  (2007),  no. 4, 503--549.





\bibitem{Pa-book} {\sc V.I.~Paulsen},
 {\it Completely Bounded Maps and Dilations},
Pitman Research Notes in Mathematics, Vol.146, New York, 1986.


\bibitem{Pi} {\sc G.~Pisier}, \emph{ Similarity Problems and Completely Bounded Maps},
Springer Lect. Notes Math., Vol.1618, Springer-Verlag, New York,
1995.







\bibitem{Po-multi} {\sc G.~Popescu},
 Multi-analytic operators and some  factorization theorems,
 {\it Indiana Univ. Math.~J.}
 {\bf 38} (1989),   693--710.










      \bibitem{Po-analytic} {\sc G.~Popescu},
      {Multi-analytic operators on Fock spaces,}
      {\it Math. Ann.} {\bf 303} (1995), 31--46.










      \bibitem{Po-poisson} {\sc G.~Popescu},
     {Poisson transforms on some $C^*$-algebras generated by isometries,}
       {\it J. Funct. Anal.} {\bf 161} (1999),  27--61.












      \bibitem{Po-holomorphic} {\sc G.~Popescu},
      {Free holomorphic functions on the unit ball of $B(\cH)^n$},
      {\it J. Funct. Anal.}  {\bf 241} (2006), 268--333.




\bibitem{Po-entropy} {\sc G.~Popescu}, Entropy and multivariable interpolation, {\it Mem. Amer. Math. Soc.}  {\bf 184} (2006), no. 868, vi+83 pp.





\bibitem{Po-pluriharmonic} {\sc G.~Popescu},
{Noncommutative transforms and free pluriharmonic functions},
 {\it Adv. Math.} {\bf 220} (2009), 831-893.

 \bibitem{Po-domains-models} {\sc G.~Popescu}, Noncommutative Berezin transforms and multivariable operator model theory, {\it  J. Funct. Anal.}  {\bf 254}  (2008),  no. 4, 1003--1057.



\bibitem{Po-domains} {\sc  G.~Popescu},
Operator theory on noncommutative domains, {\it Mem. Amer. Math. Soc.}  {\bf 205}  (2010),  no. 964, vi+124 pp.

\bibitem{Po-Berezin2} {\sc  G.~Popescu}, Berezin transforms on noncommutative varieties in polydomains, {\it  J. Funct. Anal.}  {\bf 265}  (2013),  no. 10, 2500--2552.



\bibitem{Po-Berezin1} {\sc  G.~Popescu}, Berezin transforms on noncommutative polydomains, {\it Trans. Amer. Math. Soc.}  {\bf 368}  (2016),  no. 6, 4357--4416.





     \bibitem{Po-Bohr2} {\sc G.~Popescu},   Bohr inequalities on noncommutative  polydomains, preprint.


\bibitem{Po-Toeplitz-Hyperball} {\sc G.~Popescu},  Brown-Halmos characterization of multi-Toeplitz operators associated with noncommutative hyperballs, preprint.



\bibitem{RR} {\sc M.~Rosenblum and J.~Rovnyak}, {\it Hardy classes and operator theory}, Oxford University Press-New York,  1985.




\bibitem{Sc}  {\sc I.~Schur},
%
 \"Uber Potenzreihen die im innern des Einheitshreises beschr\"ankt sind,
{\it  J. Reine Angew. Math.}
 {\bf 148} (1918), 122--145.







\bibitem{SzFBK-book} {\sc B.~Sz.-Nagy, C.~Foia\c{s}, H.~Bercovici, and L.~K\' erchy}, {\em Harmonic
Analysis of Operators on Hilbert Space}, Second edition. Revised and enlarged edition. Universitext. Springer, New York, 2010. xiv+474 pp.


\bibitem{T} {\sc O.~Toeplitz}, Zur Theorie der quadratischen und bilinearen Formen von unendlichvielen Ver\" anderlichen,  {\it Math. Ann.}  {\bf 70}  (1911),  no. 3, 351--376.

\bibitem{U} {\sc H.~Upmeier}, {\it Toeplitz operators and index theory in several complex variables}, Operator Theory: Advances and Applications  {\bf 81}, Birkhauser Verlag, Basel, 1996.
        %


       \end{thebibliography}
      \end{document}